\documentclass[11pt,american,preprint]{article}
\makeatletter
\def\ps@pprintTitle{%
 \let\@oddhead\@empty
 \let\@evenhead\@empty
 \def\@oddfoot{\centerline{\thepage}}%
 \let\@evenfoot\@oddfoot}
\makeatother

\usepackage[margin=1in]{geometry}

\usepackage[dvipsnames]{xcolor}
\usepackage{centernot}
\usepackage{lipsum}
\usepackage{changepage}

\usepackage[T1]{fontenc}
\usepackage[latin9]{luainputenc}
\usepackage{amsmath}
\usepackage{amsthm}
\usepackage{amssymb}
\usepackage{stmaryrd}
\usepackage{esint}
\usepackage{nameref}
\usepackage{hyperref} 
\usepackage{cleveref} 

\usepackage[shortlabels]{enumitem}
\usepackage{chngcntr}

\newsavebox{\foobox}

\usepackage{graphicx,amssymb}

\usepackage{array}
\newcolumntype{M}[1]{>{\centering\arraybackslash}m{#1}}

\usepackage{bbm} 
\makeatletter
\numberwithin{equation}{section}
\theoremstyle{plain}
\newtheorem{thm}{\protect\theoremname}[section]

\theoremstyle{plain*}
\newtheorem*{thm*}{\protect\theoremname}

\theoremstyle{plain}
\newtheorem{lem}[thm]{\protect\lemmaname}
  
\theoremstyle{plain*}
\newtheorem*{lem*}{\protect\lemmaname}  
  
  \theoremstyle{plain}
  \newtheorem{prop}[thm]{\protect\propositionname}
  
    \theoremstyle{plain*}
  \newtheorem*{prop*}{\protect\propositionname}

\theoremstyle{remark}

\theoremstyle{remark*}
\newtheorem*{question*}{Question} 
\theoremstyle{remark}
\newtheorem{rem}[thm]{\protect\remarkname}
\theoremstyle{remark*}
\newtheorem*{rem*}{\protect\remarkname}
\theoremstyle{remark}
\newtheorem{example}[thm]{Example}
\theoremstyle{remark*}
\newtheorem*{example*}{\protect\examplename}

\theoremstyle{plain}
\newtheorem{cor}[thm]{\protect\corollaryname}
\providecommand{\corollaryname}{Corollary}
  
\theoremstyle{definition}
\newtheorem{defn}[thm]{Definition}  
\makeatother

\theoremstyle{plain} 
\newcommand{\thistheoremname}{}
\newtheorem{genericthm}[thm]{\thistheoremname}

\newtheorem*{genericthm*}{\thistheoremname}
\newenvironment{namedthm*}[1]
  {\renewcommand{\thistheoremname}{#1}%
   \begin{genericthm*}}
  {\end{genericthm*}}

\makeatother

\usepackage{babel}
 \providecommand{\lemmaname}{Lemma}
  \providecommand{\propositionname}{Proposition}
  \providecommand{\remarkname}{Remark}
\providecommand{\theoremname}{Theorem}

\newcommand{\rlim}[1]{\mathop { \mathcal {R}\text{\text{\rm{-lim}}}}_{#1}}

\newcommand{\N}{\mathbb{N}}

\newcommand{\Z}{\mathbb{Z}}

\usepackage{fancyhdr}

\title{Strongly mixing systems are almost strongly mixing of all orders}
\date{}
\author{V.\ Bergelson \and R.\ Zelada}
\begin{document}
\maketitle
\begin{abstract}
We prove that any strongly mixing action of a countable abelian group on a probability space has higher order mixing properties. This is achieved via introducing and utilizing $\mathcal R$-limits,  a notion of convergence which is based on the classical Ramsey Theorem. $\mathcal R$-limits  are intrinsically connected with a new combinatorial notion of largeness which is similar to but has stronger properties than the classical notions of  uniform density one and IP$^*$.
While the main goal of this paper is  to establish a \emph{universal} property of  strongly mixing actions of countable abelian groups, our results,  when applied to  $\Z$-actions, offer a new  way of dealing with strongly mixing transformations. In particular, we obtain several new characterizations  of strong mixing for $\Z$-actions, including a result which can be viewed as the analogue  of the weak mixing  of all orders property established  by Furstenberg in the  course of his proof of Szemer{\'e}di's theorem.
We also demonstrate the versatility of $\mathcal R$-limits by obtaining new characterizations of  higher order weak and  mild mixing for actions of countable abelian groups.
\end{abstract}

\tableofcontents
\section{Introduction}
Let $G=(G,+)$ be a countable discrete abelian group and let $(T_g)_{g\in G}$ be a measure preserving $G$-action on a separable probability space $(X,\mathcal A,\mu)$. We will call the quadruple $(X,\mathcal A,\mu, (T_g)_{g\in G})$ a measure preserving system.  A measure preserving system $(X,\mathcal A,\mu, (T_g)_{g\in G})$ is strongly mixing (or 2-mixing) if for any $A_0,A_1\in\mathcal A$, one has
\begin{equation}
    \lim_{g\rightarrow\infty} \mu(A_0\cap T_g A_1)=\mu(A_0)\mu(A_1).
\end{equation}
The goal of this paper is to obtain new results about higher order mixing properties of stronlgy mixing actions of abelian groups. These results are motivated by the following classical problem going back to Rohlin (who formulated it   for $\Z$-actions, see \cite{rokhlin1949endomorphisms}).
\begin{namedthm*}{Rohlin's Problem}\label{11.Rohlin'sProblem}
Assume that a measure preserving system $(X,\mathcal A,\mu, (T_g)_{g\in G})$ is strongly mixing. Is it true that given any $\ell\geq 2$ the system $(X,\mathcal A,\mu, (T_g)_{g\in G})$ is $(\ell+1)$-mixing, meaning that for  any $A_0,...,A_\ell\in\mathcal A$ and any sequences 
 $(g^{(1)}_k)_{k\in\N}$,...,$(g^{(\ell)}_k)_{k\in\N}$ in $G$ satisfying:
\begin{enumerate}[(i)]
\item For any $j\in\{1,...,\ell\}$
\begin{equation}\label{1.SequenceGoingToInfty}
\lim_{k\rightarrow\infty}g^{(j)}_k=\infty
\end{equation}
\item For any distinct $i,j\in\{1,...,\ell\}$,
\begin{equation}\label{1.DifferenccesGoingToInfty}
\lim_{k\rightarrow\infty}(g^{(j)}_k-g^{(i)}_k)=\infty. 
\end{equation}
\end{enumerate}
one has
\begin{equation}\label{1.RightLimitEquation}
    \lim_{k\rightarrow\infty}\mu(A_0\cap T_{g^{(1)}_k}A_1\cap\cdots\cap T_{g^{(\ell)}_k}A_\ell)=\prod_{j=0}^\ell\mu(A_j).
\end{equation}
\end{namedthm*}

While for $\Z$-actions Rohlin's problem is still unsolved,\footnote{
The notable classes of $\Z$-actions for which it is known that 2-mixing implies mixing of all orders include
ergodic automorphisms of compact groups \cite{rokhlin1949endomorphisms}, mixing transformations with singular spectrum \cite{host1991mixing}, and mixing actions of finite rank \cite{kalikow1984twofold},  \cite{ryzhikov1993joinings}. 
It is also known that some natural actions of various locally compact groups posses the property of mixing of all orders (see, for example, \cite{marcus1978horocycle}, \cite{mozes1992mixing}, \cite{ryzhikov2000rokhlin}, \cite{fayad2016multiple}).
}
an example  for $\Z^2$-actions, due to Ledrappier, shows that, in general, mixing does not imply mixing of higher orders \cite{ledrappier1978champ}. (The reader is referred to  \cite{schmidt1995AlgebraicDynamicsBook} for more Ledrappier-type examples for $\Z^d$-actions). More precisely, Ledrappier provided an example of a strongly mixing system $(\Gamma,\mathcal B,\mu,(T^nS^m)_{(n,m)\in\Z^2})$, where $\Gamma$ is a compact abelian group, $\mathcal B$ is the $\sigma$-algebra of Borel sets in $\Gamma$, $\mu$ is the normalized Haar measure on $\Gamma$ and $T,S:\Gamma\rightarrow \Gamma$ are commuting automorphisms with the property that for some measurable set $A\subseteq \Gamma$, 
$$\mu(A\cap T^{2^n}A\cap S^{2^n}A)\centernot{ \xrightarrow[n\to\infty]{}} \mu^3(A).$$

The analysis of Ledrappier's example undertaken in \cite{arenas2008ledrappier} reveals that Ledrappier's system is "almost mixing of all orders" in the sense that, for any $\ell\in\N$, if the sequences $(g_k^{(1)})_{k\in\N}$,...,$(g_k^{(\ell)})_{k\in\N}$ in $\Z^2$ satisfy \eqref{1.SequenceGoingToInfty} and \eqref{1.DifferenccesGoingToInfty} and, in addition, the $\ell$-tuples $(g_k^{(1)},...,g_k^{(\ell)})$  avoid certain rather rarefied subsets of $\Z^{2\ell}$, the equation \eqref{1.RightLimitEquation} holds for any $A_0,...,A_\ell\in\mathcal B$ (see \cite[Theorem 3.3]{arenas2008ledrappier}). The results obtained in \cite{arenas2008ledrappier} were extended in \cite{arenas2019almost} to a rather large family of systems of algebraic origin.\\ 

In view of the results obtained in \cite{arenas2008ledrappier} and \cite{arenas2019almost},  one might wonder if it could possibly be true that, similarly to the case of Ledrappier's system,  \textit{any} strongly mixing action $(X,\mathcal A,\mu, (T_g)_{g\in G})$ of an abelian group $G$ is, in some sense,  almost mixing of all orders. The goal of this paper is to establish  a result that can be interpreted as a positive answer to this question.\\

At this point, we would like to mention that in the special case when $G=\Z$ our main theorem (\cref{1.MainResult} below) has corollaries (\cref{1.ZDiagonalResult} and \cref{1.GlobalSigma*resultForZ}) which provide new non-trivial characterizations of the notion of strong mixing in terms of the largeness of sets of the form
\begin{equation}
R^{a_1,...,a_\ell}_\epsilon(A_0,...,A_\ell)=\{n\in\Z\,|\,|\mu(A_0\cap T^{a_1n}A_1\cap\cdots\cap T^{a_\ell n}A_\ell)-\prod_{j=0}^\ell\mu(A_j)|<\epsilon\}
\end{equation}
and
\begin{equation}\label{11.MixingSetForZ}
    R_\epsilon(A_0,...,A_\ell)=\{(n_1,...,n_\ell)\in \Z^\ell\,|\,|\mu(A_0\cap T^{n_1}A_1\cap\cdots \cap T^{n_\ell}A_\ell)-\prod_{j=0}^\ell \mu( A_j)|<\epsilon\}.
\end{equation}
So, if similarly to the case of more general group actions, Rohlin's problem will turn out to have a negative answer for $G=\Z$, our results can still be interpreted as a confirmation of a weaker version of Rohlin's question.\\

Let $(X,\mathcal A,\mu,(T_g)_{g\in G})$ be a measure preserving system. Let $\ell\in\N$ and   $\epsilon>0$. For any  $A_0,...,A_\ell\in\mathcal A$ consider the set
\begin{equation}\label{11.MixingSet}
    R_\epsilon(A_0,...,A_\ell)=\{(g_1,...,g_\ell)\in G^\ell\,|\,|\mu(A_0\cap T_{g_1}A_1\cap\cdots \cap T_{g_\ell}A_\ell)-\prod_{j=0}^\ell \mu( A_j)|<\epsilon\}.
\end{equation}
Clearly, the higher is the degree of multiple mixing of the system $(X,\mathcal A,\mu, (T_g)_{g\in G})$, the more massive should the set $R_\epsilon(A_0,...,A_\ell)$ be as a subset of $G^\ell$. While, for $\ell=1$, the strong mixing property of $(X,\mathcal A,\mu, (T_g)_{g\in G})$  implies that the set $R_\epsilon(A_0,A_1)$ is cofinite, this is no longer the case for $\ell\geq 2$ even if our system $(X,\mathcal A,\mu,(T_g)_{g\in G})$ is mixing of all orders. For example, for any 3-mixing system, if $\epsilon>0$ is small enough, the set
$$R_\epsilon(A_0,A_1,A_2)=\{(g_1,g_2)\in G^2\,|\,|\mu(A_0\cap T_{g_1}A_1\cap T_{g_2}A_2)-\mu(A_0)\mu(A_1)\mu(A_2)|<\epsilon\}$$
cannot contain  pairs $(g_1,g_2)$ which are too close to the "lines" $\{(g,g)\,|\,g\in G\}$, $\{(g,0)\,|\,g\in G\}$ and $\{(0,g)\,|\,g\in G\}$.\\

In what follows we will show that for \textit{any} mixing system $(X,\mathcal A,\mu, (T_g)_{g\in G})$, the subsets of $G^\ell$ which are of the form $\mathcal R_\epsilon(A_0,...,A_\ell)$ posses a strong ubiquity 
property which we will call $\tilde\Sigma_\ell^*$ and which is quite a bit stronger than the properties of largeness associated with weakly and mildly mixing systems.  In other words, we will show that for any strongly mixing system the complement of any set of the form  $R_\epsilon (A_0,...,A_\ell)$ is very "small", giving meaning to the claim that  $(X,\mathcal A,\mu,(T_g)_{g\in G})$ is "almost strongly mixing" of all orders.
This will be achieved with the  help of \textit{$\mathcal R$-limits}, a new notion of convergence  which is based on a classical combinatorial result due to Ramsey and, as we will see, is adequate for dealing with strongly mixing systems. (In particular, we will show that the  $\tilde\Sigma_\ell^*$ property of the sets $R_\epsilon(A_0,...,A_\ell)$ implies the strong mixing of $(X,\mathcal A,\mu,(T_g)_{g\in G}$)).\\

We would like to remark that while the results that we obtain are not as sharp as those obtained in \cite{arenas2008ledrappier} and \cite{arenas2019almost}, they have the advantage of being applicable to \textit{any} strongly mixing system $(X,\mathcal A,\mu,(T_g)_{g\in G})$, where $G$ is a countable abelian group.  Moreover, as will be demonstrated in Section 6, 
the versatility of $\mathcal R$-limits allows one  to obtain new and
recover some old results pertaining to multiple recurrence properties of weakly and mildly mixing
actions of countable abelian groups. We would also like to mention that, as will be seen in Section 3, the utilization of $\mathcal R$-limits brings to life many new equivalent characterizations of strong mixing (some of which bear a strong analogy with the familiar characterizations of weak mixing via convergence in density and mild mixing via IP-convergence). \\

Before introducing the mentioned above notion of largeness for subsets of $G^\ell$, we define a related and somewhat simpler notion in $G$. 
\begin{defn}
Let $m\in\N$, let $(G,+)$ be a countable abelian group, and let $E\subseteq G$.
\begin{enumerate}
    \item We say that $E$ is a $\Sigma_m$ set if it is of the form 
    $$\{g_{k_1}^{(1)}+\cdots+g_{k_m}^{(m)}\,|\,k_1<\cdots<k_m\}$$
    where for each $j\in\{1,...,m\}$, $(g_k^{(j)})_{k\in\N}$ is a sequence in $G$ which satisfies $\lim_{k\rightarrow\infty}g_k^{(j)}=\infty$.
    \item We say that $E$ is a $\Sigma_m^*$ set if it has a non-trivial intersection with every $\Sigma_m$ set. 
\end{enumerate}
\end{defn}
\begin{rem}\label{1.TheCofiniteRemark}
\begin{enumerate}[(a)]
    \item Note that a subset of $G$  is $\Sigma_1^*$ if and only if it is cofinite. On the other hand, for any $m\geq 2$, a $\Sigma_m^*$ set does not need to be cofinite. Moreover, one can show that for each $m\geq 2$, there exists a $\Sigma_{m}^*$ set which fails to be  a $\Sigma_{n}^*$ set for each $n<m$ \cite{BerZel-Hindmanesch}. 
    \item The notion  of $\Sigma_m^*$ is similar to (but much stronger than)  the notion of IP$^*$ which has an intrinsic connection to \textit{mild} mixing and which plays an instrumental role in IP ergodic theory and in Ramsey theory (see, for example, \cite{FBook}, \cite{FKIPSzemerediLong} and \cite{berMcCuIPPolySzemeredi}). The connection between these two notions will be discussed in detail in Section 5. 
\end{enumerate}
\end{rem}

Since the sets  $R_\epsilon(A_0,...,A_\ell)$ are, by definition, subsets of $G^\ell$,  the defined above notion of $\Sigma_m^*$  has to be "upgraded"  to the subsets of the cartesian  power $G^\ell$  in order to be useful in the study of the assymptotic behavior of the \textit{multiparameter} expressions of the form 
\begin{equation}\label{1.MultiParameterExpression}
{\mu(A_0\cap T_{g_1}A_1\cap\cdots \cap T_{g_\ell}A_\ell)},\,g_1,...,g_\ell\in G.
\end{equation}
 However, it is worth noting that the family of  $\Sigma_m^*$ sets is quite adequate for dealing with "diagonal" multicorrelation sequences. In the case $G=\Z$, such  diagonal sequences have  the form
\begin{equation}\label{1.SiingleParameterell}
\mu(A_0\cap T^{a_1n}A_1\cap\cdots\cap T^{a_\ell n}A_\ell),
\end{equation}
where $a_1,...,a_\ell\in\Z$, and play an instrumental role in Furstenberg's ergodic approach to Sz{\'e}meredi's theorem (\cite{furstenberg1977Szemeredi},\cite{FBook}). 
For example, our main result (\cref{1.MainResult}), while dealing with the multiparameter expressions \eqref{1.MultiParameterExpression},
has strong  corollaries of "diagonal" nature. The following theorem (which is a version of   \cref{4.InjectiveDiagonalResult} below) is an example of a new result of this kind. Note the appearance of $\Sigma_\ell^*$ sets in the formulation. 
\begin{thm}\label{1.FinitelyGeneratedDiagonal}
Let $(G,+)$ be a countable abelian group, let $(X,\mathcal A,\mu,(T_g)_{g\in G})$ be a strongly mixing  system, and let the homomorphisms $\phi_1,...,\phi_\ell:G\rightarrow G$ be such that for any $j\in\{1,...,\ell\}$, $\ker(\phi_j)$ is finite and for any $i\neq j$, $\ker(\phi_j-\phi_i)$ is also finite. Then   for any $A_0,...,A_\ell\in\mathcal A$ and any $\epsilon>0$ the set
\begin{equation}
R_\epsilon^{\phi_1,...,\phi_\ell}(A_0,...,A_\ell)
=\{g\in G\,|\,|\mu(A_0\cap T_{\phi_1(g)}A_1\cap \cdots \cap T_{\phi_\ell(g)}A_\ell)-\prod_{j=0}^\ell\mu(A_j)|<\epsilon\}
\end{equation}
is $\Sigma_\ell^*$.
\end{thm}
When $G$ is finitely generated, \cref{1.FinitelyGeneratedDiagonal} has a stronger version (\cref{4.FinitelyGeneratedEquivalence}), which in the case $G=\Z$ can be formulated as follows.
\begin{thm}\label{1.ZDiagonalResult}
Let $(X,\mathcal A,\mu, T)$ be a measure preserving system, let $\ell\in\N$, and let $a_1,...,a_\ell$  be  distinct non-zero integers. Then $T$ is  strongly mixing if and only if for any $A_0,...,A_\ell\in\mathcal A$ and any $\epsilon>0$, the set 
\begin{equation}\label{1.ZDiagonalEqInLemma}
R^{a_1,...,a_\ell}_\epsilon(A_0,...,A_\ell)=\{n\in\Z\,|\,|\mu(A_0\cap T^{a_1n}A_1\cap\cdots\cap T^{a_\ell n}A_\ell)-\prod_{j=0}^\ell\mu(A_j)|<\epsilon\}
\end{equation}
is $\Sigma_\ell^*$.\footnote{For a related result see \cite[Theorem 1.11]{BerZel-ItteratedDifferenceDiophantine}. See also \cite{KuangYeDeltaMixing}.}
\end{thm}
\begin{rem}
One can view  \cref{1.ZDiagonalResult} as a strongly mixing analogue of two theorems due to Furstenberg  which  pertain to weak and mild mixing (see Theorems 4.11 and 9.27 in \cite{FBook}). 
The first of these two theorems states that the assumption that $(X,\mathcal A,\mu, T)$ is weakly mixing, implies (and is implied by the fact) that the sets $R^{a_1,...,a_\ell}_\epsilon(A_0,...,A_\ell)$ defined in \eqref{1.ZDiagonalEqInLemma} have uniform density one. The second one states that the assumption that  $(X,\mathcal A,\mu,T)$ is mildly mixing implies (and is implied by) the IP$^*$ property of the  sets $R^{a_1,...,a_\ell}_\epsilon(A_0,...,A_\ell)$. These theorems are instrumental for the proofs of ergodic  Szemer{\'e}di \cite{furstenberg1977Szemeredi} and IP-Szemer{\'e}di \cite{FKIPSzemerediLong} theorems. 
\end{rem}
Note that, for $\ell=1$, both diagonal (see \eqref{1.SiingleParameterell}) and multiparameter (see \eqref{1.MultiParameterExpression}) multicorrelation sequences  reduce to the classical expression $\mu(A_0\cap T_{g}A_1)$. The following theorem (which is a very special case of stronger results to be established in this paper) shows that, even in the rather degenerated case $\ell=1$, $\Sigma_m^*$ sets provide a new characterization for the notion of strong mixing for actions of abelian groups. 
\begin{thm}\label{1.MixingCharacteriazation}
Let $(G,+)$ be a countable abelian group and let $(X,\mathcal A,\mu, (T_g)_{g\in G})$ be a measure preserving system. The following statements are equivalent:
\begin{enumerate}[(i)]
    \item $(T_g)_{g\in G}$ is strongly mixing.
    \item For any $m\in\N$, any $\epsilon>0$ and any $A_0,A_1\in\mathcal A$, the set $$R_\epsilon(A_0,A_1)=\{g\in G\,|\,|\mu(A_0\cap T_g A_1)-\mu(A_0)\mu(A_1)|<\epsilon\}$$
    is $\Sigma_m^*$ in $G$.
    \item There exists an $m\in\N$ such that for any $\epsilon>0$ and any $A_0,A_1\in\mathcal A$, the set $R_\epsilon(A_0,A_1)$ is $\Sigma_m^*$ in $G$.
\end{enumerate}
\end{thm}
We are moving now to define the modified versions of $\Sigma_m$ and $\Sigma_m^*$ sets which will be instrumental in our dealing with the multiple mixing properties of strongly mixing systems.
\begin{defn}
Let $(G,+)$ be a countable abelian group and let $(g_k)_{k\in\N}$ and $(h_k)_{k\in\N}$ be two sequences in $G$. We say that $(g_k)_{k\in\N}$ and $(h_k)_{k\in\N}$ \textbf{grow apart} if $\lim_{k\rightarrow\infty}(g_k-h_k)=\infty$. 
\end{defn}
\begin{defn}
Let  $(G,+)$ be a countable abelian group, let $d\in\N$ and let $(\textbf g_k)_{k\in\N}=(g_{k,1},...,g_{k,d})_{k\in\N}$ be a sequence in $G^d$. We say that $(\textbf g_k)_{k\in\N}$ is \textbf{non-degenerated} if for each $j\in\{1,...,d\}$, $${\lim_{k\rightarrow\infty}g_{k,j}=\infty}.$$
\end{defn}
\begin{defn}\label{1.SigmaTildeDefn}
Let $d,m\in\N$ and  let $(G,+)$ be a countable abelian group.
\begin{enumerate}
    \item We say that $E\subseteq G^d$ is a $\tilde\Sigma_m$ set if it is of the form 
    $$\{\textbf g_{k_1}^{(1)}+\cdots+\textbf g_{k_m}^{(m)}\,|\,k_1<\cdots<k_m\}$$
    where for each $j\in\{1,...,m\}$, $(\textbf g_k^{(j)})_{k\in\N}=(g_{k,1}^{(j)},...,g_{k,d}^{(j)})_{k\in\N}$ is a non-degenerated  sequence in $G^d$ and for any distinct  $t,t'\in\{1,...,d\}$ the sequences $(g_{k,t}^{(j)})_{k\in\N}$ and $(g_{k,t'}^{(j)})_{k\in\N}$ are growing apart. (Note that if $d=1$, then $E\subseteq G$ is a $\Sigma_m$ set if and only if it is a $\tilde\Sigma_m$ set.) 
    \item We say that $E\subseteq G^d$ is a $\tilde \Sigma_m^*$ set if it has a non-trivial intersection with every $\tilde \Sigma_m$ set in $G^d$. 
\end{enumerate}
\end{defn}
\begin{rem}
The main difference between $\tilde\Sigma_m$ sets and $\Sigma_m$ sets is that $\tilde\Sigma_m$ sets are  subsets of cartesian powers of $G$ and have the built-in feature which guarantees that, asymptotically, the elements of $\tilde\Sigma_m$ sets stay away from "degenerated" subsets such as, for example, the following subsets of $G^3$:  $\{(g,g,g)\,|\,g\in G\}$, $\{(g,2g,0)\,|\,g\in G\}$ and  $\{(g,g,h)\,|\,g,h\in G\}$. 
\end{rem}
The following theorem, which is a corollary of \cref{1.MainResult} below, demonstrates the relevance of $\tilde\Sigma_m$ sets for dealing with mixing of higher orders.
\begin{thm}\label{1.GlobalSigma*result}
Let $(G,+)$ be a countable abelian group and let $(X,\mathcal A,\mu, (T_g)_{g\in G})$ be a measure preserving system. The following statements are equivalent:
\begin{enumerate}[(i)]
    \item $(T_g)_{g\in G}$ is strongly mixing.
    \item For any $\ell\in\N$, any $A_0,...,A_\ell\in\mathcal A$ and any $\epsilon>0$, the set 
     \vspace{-0.5cm}
     $$R_\epsilon(A_0,...,A_\ell)=\{(g_1,...,g_\ell)\in G^\ell\,|\,|\mu(A_0\cap T_{g_1}A_1\cap\cdots \cap T_{g_\ell}A_\ell)-\prod_{j=0}^\ell \mu( A_j)|<\epsilon\} \vspace{-0.5cm}$$
    is $\tilde\Sigma_\ell^*$ in $G^\ell$.
    \item There exists an $\ell\in\N$ such that for  any $A_0,...,A_\ell\in\mathcal A$ and any $\epsilon>0$, the set $R_\epsilon(A_0,...,A_\ell)$ is $\tilde\Sigma_\ell^*$ in $G^\ell$.
\end{enumerate}
\end{thm}
We take the liberty of stating explicitly the following special case of  \cref{1.GlobalSigma*result} to stress the applicability of the aparatus developed in this paper to $\Z$-actions.
\begin{cor}\label{1.GlobalSigma*resultForZ}
Let $(X,\mathcal A,\mu, T)$ be a measure preserving system. The following statements are equivalent:
\begin{enumerate}[(i)]
    \item $T$ is strongly mixing.
    \item For any $\ell\in\N$, any $A_0,...,A_\ell\in\mathcal A$ and any $\epsilon>0$, the set 
     \vspace{-0.5cm}
     $$R_\epsilon(A_0,...,A_\ell)=\{(n_1,...,n_\ell)\in \Z^\ell\,|\,|\mu(A_0\cap T^{n_1}A_1\cap\cdots \cap T^{n_\ell}A_\ell)-\prod_{j=0}^\ell \mu( A_j)|<\epsilon\} \vspace{-0.5cm}$$
    is $\tilde\Sigma_\ell^*$ in $\Z^\ell$.
    \item There exists an $\ell\in\N$ such that for  any $A_0,...,A_\ell\in\mathcal A$ and any $\epsilon>0$, the set $R_\epsilon(A_0,...,A_\ell)$ is $\tilde\Sigma_\ell^*$ in $\Z^\ell$.
\end{enumerate}
\end{cor}
We introduce now the notion of convergence based on the classical Ramsey Theorem that is behind the proof of \cref{1.GlobalSigma*result}. Given $m\in\N$ and an infinite set $S\subseteq \N$, we denote by $S^{(m)}$ the family of all $m$-element subsets of $S$. When writing $\{k_1,...,k_m\}\in S^{(m)}$, we will always assume that $k_1<\cdots<k_m$. 
\begin{thm}[Ramsey's Theorem]\label{1.Ramsey}
Let $r,m\in\N$ and let $C_1,...,C_r\subseteq\N^{(m)}$ be such that 
\begin{equation}\label{1.PartitionOfN}
\N^{(m)}=\bigcup _{j=1}^r C_j.
\end{equation}
Then there exists $j_0\in\{1,...,r\}$ and an infinite subset $S\subseteq\N$, satisfying 
$S^{(m)}\subseteq C_{j_0}$.
\end{thm}
\begin{rem}\label{1.RamseyProofRemark}
It is easy to see that \cref{1.Ramsey} can be formulated in the  following  equivalent form  that will be frequently used in the sequel:\\
\begin{adjustwidth}{0.5cm}{0.5cm}
\textit{Let $r,m\in\N$, let $P$ be an infinite subset of $\N$  and let $C_1,...,C_r\subseteq\N^{(m)}$ be such that
\begin{equation}
P^{(m)}\subseteq\bigcup _{j=1}^r C_j.
\end{equation}
Then there exists $j_0\in\{1,...,r\}$ and an infinite subset $S\subseteq P$, satisfying 
$S^{(m)}\subseteq C_{j_0}$.}
\end{adjustwidth}
\end{rem}
\begin{defn}
Let $(X,d)$ be a compact metric space, let $x\in X$, let $(x_\alpha)_{\alpha\in \N^{(m)}}$ be an ${\N^{(m)}\text{-sequence}}$ in $X$ and let $S$ be an infinite subset of $\N$. We write
\begin{equation}
    \mathop{\mathcal R\text{-lim}}_{\alpha\in S^{(m)}}x_\alpha=x
\end{equation}
if for every $\epsilon>0$, there exists $\alpha_0\in \N^{(m)}$ such that for any  $\alpha\in S^{(m)}$ satisfying  
${\min \alpha>\max \alpha_0,}$
one has
$$d(x_\alpha, x)<\epsilon.$$
\end{defn}
The following theorem can be viewed as  a version of Bolzano-Weierstrass theorem for ${\mathcal R\text{-convergence}}$. It follows from \cref{1.Ramsey} with the help of a diagonalization argument.
\begin{thm}\label{1.RBolzanoWierstrass}
Let $(X,d)$ be a compact metric space and let $(x_\alpha)_{\alpha\in\N^{(m)}}$ be an $\N^{(m)}$-sequence in $X$. Then for any infinite set $S_1\subseteq \N$ there exists an $x\in X$ and an infinite set $S\subseteq S_1$ such that
\begin{equation}
    \rlim{\alpha\in S^{(m)}}x_{\alpha}=x.
\end{equation}
\end{thm}

\begin{rem}\label{1.ItteratedLimitsRemark}
Let $(x_\alpha)_{\alpha\in\N^{(m)}}$ be an $\N^{(m)}$-sequence in a compact metric space $(X,d)$. The introduced above $\mathcal R$-limits have an intrinsic connection with the iterated limits of the form 
\begin{equation}\label{1.ItteratedLimitsExistence}
\lim_{j_1\rightarrow\infty}\cdots\lim_{j_m\rightarrow\infty}x_{\{k_{j_1},...,k_{j_m}\}}.\footnote{
Cf. \cite{sucheston1959sequences} and \cite{lorentz1960remark}.
}
\end{equation} 
The goal of this extended remark is to clarify this connection.
\begin{enumerate}[(a)]
\item Using the compactness of $X$, one can show with the help of a diagonalization argument that for any increasing sequence $(k_j)_{j\in\N}$, there exists a subsequence $(k_j')_{j\in\N}$ for which all the limits in \eqref{1.ItteratedLimitsExistence} exist. 
\item  By \cref{1.RBolzanoWierstrass}, there exists an increasing sequence of natural numbers $(k_j)_{j\in\N}$ so that for $S=\{k_j\,|\,j\in\N\}$, $\rlim{\alpha\in S^{(m)}}x_\alpha$ exists. Let $(k_j')_{j\in\N}$ be the subsequence of $(k_j)_{j\in\N}$ which  is guaranteed to exist by (a). Letting $S_1=\{k_j'\,|\,j\in\N\}$, we have 
\begin{equation}\label{1.ItteratedLimits}
\rlim{\alpha\in S_1^{(m)}}x_\alpha=\lim_{j_1\rightarrow\infty}\cdots\lim_{j_m\rightarrow\infty}x_{\{k'_{j_1},...,k'_{j_m}\}}.
\end{equation}
\item When $X=\{1,...,r\}$, one can use (a) to prove \cref{1.Ramsey}. Let $r,m\in\N$ and consider a partition $\N^{(m)}=\bigcup_{j=1}^r C_j$. Let $(x_\alpha)_{\alpha\in\N^{(m)}}$ be defined by $x_\alpha=j$ if $\alpha\in C_j$. For some increasing sequence $(k_j)_{j\in\N}$ in $\N$  there exists a $j_0\in\{1,...,r\}$ such that 
$$\lim_{j_1\rightarrow\infty}\cdots\lim_{j_m\rightarrow\infty}x_{\{k_{j_1},...,k_{j_m}\}}=j_0.$$ 
By using  a diagonalization argument, we obtain a subsequence $(k_j')_{j\in\N}$ of $(k_j)_{j\in\N}$ with the property that $x_{\{k_{j_1}',...,k_{j_m}'\}}=j_0$ for any $j_1<\cdots<j_m$.  Now let $S=\{k_j'\,|\,j\in\N\}$. It follows that $S^{(m)}\subseteq C_{j_0}$. 
\end{enumerate}
\end{rem}
Before formulating our main result, we need two more definitions. 
\begin{defn}
Let $m\in\N$ and let $(G,+)$ be a countable abelian group. For any sequence $(\textbf g_k)_{k\in\N}=(g_{k,1},...,g_{k,m})_{k\in\N}$ and any $\alpha=\{k_1,...,k_m\}\in\N^{(m)}$ we let 
\begin{equation}\label{1.gAlphaDefn}
g_\alpha=\sum_{j=1}^m g_{k_j,j}=g_{k_1,1}+g_{k_2,2}+\cdots+g_{k_m,m},
\end{equation}
where $k_1<\cdots<k_m$. 
\end{defn}
\begin{defn}
Let $m\in\N$, let $(G,+)$ be a countable abelian group and let 
$$(\textbf g_k)_{k\in\N}=(g_{k,1},...,g_{k,m})_{k\in\N}\text{ and }(\textbf h_k)_{k\in\N}=(h_{k,1},...,h_{k,m})_{k\in\N}$$
be sequences in $G^m$. We  say that  $(\textbf g_k)_{k\in\N}$ and $(\textbf h_k)_{k\in\N}$ are \textbf{essentially distinct} if for each $t\in\{1,...,m\}$, $(g_{k,t})_{k\in\N}$ and $(h_{k,t})_{k\in\N}$ grow apart (i.e. $\lim_{k\rightarrow\infty}(g_{k,t}-h_{k,t})=\infty$).  
\end{defn}
\begin{rem}
The following observation indicates the natural  connection between non-degenerated, essentially distinct sequences in $G^m$ and $\tilde\Sigma_m$ sets. Let $d,m\in\N$ and let $(G,+)$ be a countable abelian group. Then for any non-degenerated and essentially distinct sequences 
$$(\textbf g_k^{(j)})_{k\in\N}=(g_{k,1}^{(j)},...,g_{k,m}^{(j)})_{k\in\N},\,j\in\{1,...,d\},$$
in $G^m$, the set 
\begin{multline*}
\{(g_\alpha^{(1)},...,g_\alpha^{(d)})\,|\,\alpha\in\N^{(m)}\}=\{(g^{(1)}_{k_1,1}+\cdots+g^{(1)}_{k_m,m},...,g^{(d)}_{k_1,1}+\cdots+g^{(d)}_{k_m,m})\,|\,k_1<\cdots<k_m\}\\
=\{(g^{(1)}_{k_1,1},...,g^{(d)}_{k_1,1})+\cdots+(g^{(1)}_{k_m,m},...,g^{(d)}_{k_m,m})\,|\,k_1<\cdots<k_m\}
\end{multline*}
is a $\tilde\Sigma_m$ set in $G^d$. 
\end{rem}
We are ready now to formulate our main result (it appears as \cref{3.MainResult} in Section 3). It incorporates some of the characterizations of strongly mixing systems which were mentioned above. 
\begin{thm}\label{1.MainResult}
Let $\ell\in\N$, let $(G,+)$ be a countable abelian group and let $(X,\mathcal A,\mu, (T_g)_{g\in G})$ be a  measure preserving system. 
The following statements are equivalent:
\begin{enumerate}[(i)]
    \item  $(T_g)_{g\in G}$ is strongly mixing.
     \item For any non-degenerated and essentially distinct sequences 
    $(\textbf g_k^{(j)})_{k\in\N}$, $j\in\{1,...,\ell\},$
   in $G^{(\ell)}$, there exists an infinite $S\subseteq\N$ such that for any  $A_0,...,A_\ell\in\mathcal A$,
    \begin{equation}
        \rlim{\alpha\in S^{(\ell)}}\mu(A_0\cap T_{ g^{(1)}_\alpha}A_1\cap \cdots\cap T_{ g^{(\ell)}_\alpha}A_\ell)=\prod_{j=0}^\ell\mu(A_j).
    \end{equation}
    More explicitly, if 
    $$(\textbf g^{(j)}_k)_{k\in\N}=(g^{(j)}_{k,1},...,g^{(j)}_{k,\ell})_{k\in\N},$$ for each $j\in\{1,...,\ell\},$
    then
    $$\rlim{\{k_1,...,k_\ell\}\in S^{(\ell)}}\mu(A_0\cap T_{ g_{k_1,1}^{(1)}+\cdots+ g_{k_\ell,\ell}^{(1)}}A_1\cap\cdots\cap T_{  g_{k_1,1}^{(\ell)}+\cdots+ g_{k_\ell,\ell}^{(\ell)}}A_\ell)=\prod_{j=0}^\ell \mu(A_j).$$
    \item For any $\epsilon>0$ and any $A_0,...,A_\ell\in\mathcal A$, the set $$R_\epsilon(A_0,...,A_\ell)=\{(g_1,...,g_\ell)\in G^\ell\,|\,|\mu(A_0\cap T_{g_1}A_1\cap\cdots \cap T_{g_\ell}A_\ell)-\prod_{j=0}^\ell \mu( A_j)|<\epsilon\}$$ 
    is $\tilde \Sigma_\ell^*$ in $G^\ell$.
      \item For any $\epsilon>0$ and any $A_0,A_1\in\mathcal A$, the set $R_\epsilon(A_0,A_1)$ is $\Sigma_\ell^*$ in $G$.
    \end{enumerate}
\end{thm}

The structure of this paper is as follows. In Section 2, some basic facts about couplings of probability spaces are reviewed and some auxiliary results which will be needed in Sections 3 and 6 are established. In Section 3, we prove our main result, \cref{1.MainResult} (=\cref{3.MainResult}). In Section 4, we  derive some diagonal results for strongly mixing systems. In Section 5, we describe the largeness properties of $\tilde\Sigma_m^*$ sets and, more specifically, of the sets $R_\epsilon(A_0,...,A_\ell)$. We also juxtapose the properties of $\tilde\Sigma_m^*$ sets with those of $\tilde{{\text{IP}}}\rm{^*}$ sets and  sets of uniform density one which are  characteristic, correspondingly, of mild and weak mixing.  In Section 6, 
we utilize the methods developed in Sections 2 and 5 to obtain analogues of \cref{1.MainResult}  for mildly and weakly mixing systems.   
\begin{rem}
Throughout this paper, we will be tacitly assuming that the measure preserving systems $(X,\mathcal A,\mu,(T_g)_{g\in G})$ that we are working with are \textit{regular} meaning that the underlying probability space  $(X,\mathcal A,\mu)$ is regular (i.e. $X$ is a  compact metric space and $\mathcal A=\text{Borel}(X)$). Note that this assumption can be made without loss of generality since every separable measure preserving system is equivalent to a regular one (see for instance, \cite[Proposition 5.3]{FBook}).
\end{rem}
\section{Some auxiliary facts involving couplings and $\mathcal R$-limits}
In this section we review some basic facts about couplings of probability spaces and establish some auxiliary results which will be needed in Section 3 and Section 6.
\begin{defn}\label{2.CouplingDefn}
Let $N\in\N$. Given regular probability spaces  $\textbf X_j=(X_j,\mathcal A_j,\mu_j)$, $j\in\{1,...,N\}$, a \textbf{coupling} of  $\textbf X_1,...,\textbf X_N$ is a  Borel probability measure $\lambda$ defined on the measurable space 
$$(\prod_{j=1}^NX_j,\bigotimes_{j=1}^N \mathcal A_j)$$ 
and having the property that for any $j\in\{1,...,N\}$ and any $A\in\mathcal A_j$, $\lambda(\pi_j^{-1}(A))=\mu_j(A)$, where  $\pi_j:\prod_{i=1}^NX_i\rightarrow X_j$ is the projection map onto the j-th coordinate of $\prod_{j=1}^N X_j$.\footnote{
A coupling is just a \textit{joining} of the trivial measure preserving systems $(X_j,\mathcal A_j,\mu_j, \text{Id}_j)$, $j\in\{1,...,N\}$, where $\text{Id}_j:X_j\rightarrow X_j$ denotes the identity map on $X_j$.
}
\end{defn}
We will let $\mathcal C(\textbf X_1,...,\textbf X_N)$ denote the set of all couplings of $\textbf X_1,...,\textbf X_N$. $\mathcal C(\textbf X_1,...,\textbf X_N)$ is a closed subspace of the set of all probability Borel measures on $\prod_{j=1}^NX_j$ endowed with the ${\text{weak-*}}$ topology. With this topology, $\mathcal C(\textbf X_1,...,\textbf X_N)$ is a compact metrizable space. Given a sequence $(\lambda_k)_{k\in\N}$ in $\mathcal C(\textbf X_1,...,\textbf X_N)$,
$$\lambda_k\xrightarrow[k\rightarrow\infty]{} \lambda$$
if and only if for any $A_1\in\mathcal A_1$,...,$A_N\in\mathcal A_N$,
$$\lambda_k(A_1\times\cdots\times A_N)\xrightarrow[k\rightarrow\infty]{} \lambda(A_1\times\cdots\times A_N).$$

The following proposition follows immediately from the compactness of $\mathcal C(\textbf X_1,...,\textbf X_N)$ and \cref{1.RBolzanoWierstrass}. 
\begin{prop}\label{2.EquivalentLimits}
Let   $\textbf X_j=(X_j,\mathcal A_j,\mu_j)$, $j\in\{1,...,N\}$, be  regular probability spaces. For any $m\in\N$, any infinite $S\subseteq \N$ and any $\N^{(m)}$-sequence $(\lambda_\alpha)_{\alpha\in\N^{(m)}}$ in  $\mathcal C(\textbf X_1,...,\textbf X_N)$,
$$\rlim{\alpha\in S^{(m)}}\lambda_\alpha=\lambda$$
if and only if for any $A_1\in\mathcal A_1$,...,$A_N\in\mathcal A_N$,
$$\rlim{\alpha\in S^{(m)}}\lambda_\alpha(A_1\times \cdots\times A_N)=\lambda(A_1\times\cdots \times A_N).$$
\end{prop}
Our next goal is to stablish a useful criterion for mixing of higher orders (\cref{2.MainResult}). First, we need a definition and two lemmas. 
\begin{defn}
Let $(Z,\mathcal D,\lambda)$ be a regular probability space and let, for each $k\in\N$, $T_k:Z\rightarrow Z$ be a measure preserving transformation.  The sequence  $(T_k)_{k\in\N}$ has the mixing property if for every $A_0,A_1\in\mathcal D$,
$$\lim_{k\rightarrow\infty}\lambda(A_0\cap T_k^{-1}A_1)=\lambda(A_0)\lambda(A_1).$$
\end{defn}
\begin{rem}
\begin{enumerate}[(a)]
\item If each of the transformations $T_k$, $k\in\N$, is invertible, $(T_k)_{k\in\N}$ has   the mixing property if and only if $(T_k^{-1})_{k\in\N}$  has the mixing property.
\item $(T_k)_{k\in\N}$ has the mixing property if and only if for any $f,g\in L^2(\mu)$, $$\lim_{k\rightarrow\infty}\int_X fT_kg\text{d}\mu=\int_Xf\text{d}\mu\int_Xg\text{d}\mu.$$
\end{enumerate}
\end{rem}
\begin{lem}\label{2.PropTheIndependentjoining}
Let $\textbf X=(X,\mathcal A,\mu)$ and $\textbf Y=(Y,\mathcal B,\nu)$ be regular probability spaces. For each $k\in\N$, let $T_k:Y\rightarrow Y$ be a measure preserving transformation, and assume that the sequence $(T_k)_{k\in\N}$ has the mixing property. Let $\lambda_0$ be a coupling of $\textbf X$ and $\textbf Y$. Assume that $\lambda$ is a probability measure on $\mathcal A\otimes\mathcal B$ such that for any    $A\in\mathcal A$ and $B\in\mathcal B$ one has
\begin{equation}\label{2.LimitCondition}
    \lim_{k\rightarrow\infty}\lambda_0((\text{Id}\times T_k^{-1})(A\times B))=\lambda(A\times B).
\end{equation}
Then $\lambda=\mu\otimes\nu$.
\end{lem}
\begin{proof}
Note that it suffices to show that for any  $A\in \mathcal A$ and $B\in \mathcal B$,
\begin{equation}\label{2.StatementToProveLemma}
\lambda(A\times B)=\mu(A)\nu(B).
\end{equation}
Fix $A\in\mathcal A$ and $B\in\mathcal B$. Since $\mathbbm  1_A\otimes \mathbbm  1_B=(\mathbbm  1_A\otimes\mathbbm 1_Y)(\mathbbm 1_X\otimes \mathbbm  1_B)$, we have by \eqref{2.LimitCondition} that 
\begin{multline}\label{2.FirstEquationLemma}
    \int_{X\times Y}(\mathbbm  1_A\otimes\mathbbm 1_Y) (\mathbbm 1_X\otimes \mathbbm  1_B)\text{d}\lambda=\lambda(A\times B)=\\
    \lim_{k\rightarrow\infty}\lambda_0((\text{Id}\times T_k^{-1})(A\times B))
    =\lim_{k\rightarrow\infty} \int_{X\times Y}(\text{Id}\times T_k)(\mathbbm  1_A\otimes\mathbbm 1_Y)(\text{Id}\times T_k) (\mathbbm 1_X\otimes \mathbbm  1_B)\text{d}\lambda_0.
\end{multline}
Note that $(\text{Id}\times T_k) (\mathbbm  1_A\otimes \mathbbm 1_Y)=\mathbbm  1_A\otimes \mathbbm 1_Y$ and, if we regard $\mathcal B$ as a sub $\sigma$-algebra of $\mathcal A\otimes\mathcal B$, $\lambda_0|_{\mathcal B}=\nu$. The right-most expression  in \eqref{2.FirstEquationLemma} equals 
\begin{multline}\label{2.SecondEquationLemma}
\lim_{k\rightarrow\infty}\int_{X\times Y} (\mathbbm  1_A\otimes \mathbbm 1_Y)( \mathbbm 1_X\otimes T_k\mathbbm  1_B)\text{d}\lambda_0=\lim_{k\rightarrow\infty}\int_{X\times Y} \mathbb E(\mathbbm  1_A\otimes \mathbbm 1_Y\,|\,\mathcal B)( \mathbbm 1_X\otimes T_k\mathbbm  1_B)\text{d}\lambda_0\\
=\lim_{k\rightarrow\infty}\int_Y \mathbb E(\mathbbm  1_A\otimes \mathbbm 1_Y\,|\,\mathcal B)T_k\mathbbm  1_B\text{d}\nu.
\end{multline}
where $\mathbb E(\mathbbm  1_A\otimes \mathbbm 1_Y\,|\,\mathcal B)$ denotes the conditional expectation of $\mathbbm  1_A\otimes \mathbbm 1_Y$ with respect to $\mathcal B$.\\
But $(T_k)_{k\in\N}$ has the mixing property, so the right-most expression in \eqref{2.SecondEquationLemma} equals
\begin{equation}\label{2.FinalExpressionInSmallLemma}
\int_Y\mathbb E(\mathbbm  1_A\otimes \mathbbm 1_Y\,|\,\mathcal B)\text{d}\nu\int_Y \mathbbm  1_B\text{d}\nu=\lambda(A\times B).
\end{equation}
By noting that 
$$\int_Y\mathbb E(\mathbbm  1_A\otimes \mathbbm 1_Y\,|\,\mathcal B)\text{d}\nu=\int_{X\times Y}(\mathbbm  1_A\otimes \mathbbm 1_Y)\text{d}\lambda_0=\int_X \mathbbm  1_A\text{d}\mu,$$
we have that \eqref{2.FinalExpressionInSmallLemma} equals $\mu(A)\nu(B)$.
\end{proof}

\begin{lem}\label{2.DecomposingRlimLemma}
Let $m\in\N$, let $(X,d)$ be a compact metric space and let $(x_\alpha)_{\alpha\in\N^{(m+1)}}$ be an $\N^{(m+1)}$-sequence  in $X$. Assume that there  exists an infinite $S\subseteq \N$ with the properties (a) for some $x\in X$, $\rlim{\alpha\in S^{(m+1)}}x_\alpha=x$ and (b) for each $k\in S$ there exists $y_k\in X$ such that 
$$\rlim{\alpha\in S^{(m)},\,k<\min \alpha}x_{\{k\}\cup\alpha}=y_k.$$
Then
$$\lim_{k\rightarrow\infty,\,k\in S}\rlim{\alpha\in S^{(m)},\,k<\min \alpha}x_{\{k\}\cup \alpha}=\lim_{k\rightarrow\infty,\,k\in S}y_k=\rlim{\alpha\in S^{(m+1)}}x_\alpha.$$
\end{lem}
\begin{proof}
Let $\epsilon>0$. Note that  (1) there exists $k_0\in S$ such that for any  $\alpha\in S^{(m+1)}$ with $k_0\leq \min \alpha$, $d(x_\alpha, x)<\frac{\epsilon}{2}$ and (2) for any $k\in S$ there exists an $\alpha_k\in S^{(m)}$ such that for any $\alpha\in S^{(m)}$ with 
$\min \alpha >\max (\alpha_k\cup\{k\})$, $d(x_{\{k\}\cup \alpha},y_k)<\frac{\epsilon}{2}$. It follows that for any $k\in S$ with  $k\geq k_0$ and any $\alpha\in S^{(m)}$ with $\min \alpha>\max(\alpha_k\cup\{k\})$, $d(y_k,x)<d(x_{\{k\}\cup\alpha},y_k)+d(x_{\{k\}\cup \alpha},x)<\epsilon$. Since $\epsilon>0$ was arbitrary,
$$\lim_{k\rightarrow\infty,\,k\in S}y_k=x=\rlim{\alpha\in S^{(m+1)}}x_\alpha.$$
\end{proof}
\begin{rem}\label{2.ExistenceOfS}
Let $m\in\N$ and let $(x_\alpha)_{\alpha\in\N^{(m+1)}}$ be an $\N^{(m+1)}$-sequence in a compact metric space $X$. By applying \cref{1.RBolzanoWierstrass} first to the  $\N^{(m)}$-sequence  $(\omega_\alpha)_{\alpha\in\N^{(m)}}=((x_{\{k\}\cup\alpha})_{k\in\N})_{\alpha\in\N^{(m)}}$ in $X^\N$ (here $x_{\{k\}\cup\alpha}=x_0$ for some fixed $x_0\in X$, whenever $k\geq\min \alpha$),  and then to the $\N^{(m+1)}$-sequence $(x_\alpha)_{\alpha\in\N^{(m+1)}}$, we obtain an infinite set $S\subseteq \N$ for which (a) and (b) in the statement of \cref{2.DecomposingRlimLemma} hold.  A similar reasoning shows that one can pick $S$ to be a subset of any prescribed in advance infinite set $S_1\subseteq \N$.
\end{rem}
\begin{rem}\label{2.ItteratedLimits}
In \cref{1.ItteratedLimitsRemark},(c), we indicated how the utilization of  iterated limits
\begin{equation*}\label{2.MultiLimitExpression}
\lim_{j_1\rightarrow\infty}\cdots\lim_{j_m\rightarrow\infty}x_{\{k_{j_1},...,k_{j_m}\}}
\end{equation*}
leads to a proof of Ramsey's theorem (\cref{1.Ramsey}). In this remark, we show that \cref{2.DecomposingRlimLemma} and \cref{2.ExistenceOfS} (which are corollaries of Ramsey's Theorem) imply  that for any  infinite set $S_1\subseteq \N$ and any  $\N^{(m)}$-sequence $(x_\alpha)_{\alpha\in\N^{(m)}}$ in a compact metric space $X$, there exists an increasing sequence $(k_j)_{j\in\N}$ in $S_1$ such that for $S=\{k_j\,|\,j\in\N\}$ each of the limits in the formula
\begin{equation*}\label{2.RLim=LimLim}
\rlim{\alpha\in S^{(m)}}x_\alpha=\lim_{j_1\rightarrow\infty}\cdots\lim_{j_m\rightarrow\infty}x_{\{k_{j_1},...,k_{j_m}\}}
\end{equation*}
exist. The proof is by induction on $m\in\N$. When $m=1$, the result follows form the compactness of $X$. Now let $m>1$ and let $S_1$ be an infinite subset of $\N$. By \cref{2.ExistenceOfS} and \cref{2.DecomposingRlimLemma}, there exists an increasing sequence $(k_j)_{j\in\N}$ in $S_1$ such that for $S=\{k_j\,|\,j\in\N\}$,
$$\rlim{\alpha\in S^{(m)}}x_\alpha=\lim_{j\rightarrow\infty}\rlim{\alpha\in S^{(m-1)}}x_{\{k_j\}\cup \alpha}.$$
The result now follows from the inductive hypothesis applied to the infinite set $S\subseteq \N$ and the $\N^{(m-1)}$-sequence $((x_{\{k\}\cup \alpha})_{k\in\N})_{\alpha\in\N^{(m-1)}}$ in the compact metric spacce $X^\N$.
\end{rem}
The following proposition provides a useful technical tool for  establishing higher order mixing properties of measure preserving systems. 
It will be instrumental in Section 3 for dealing with strongly mixing systems and in Section 6, where we will focus on mildly and weakly mixing systems.
\begin{prop}\label{2.MainResult}
Let $(G,+)$ be a countable abelian group, let $(X,\mathcal A,\mu, (T_g)_{g\in G})$ be a measure preserving system, let $\ell\in\N$ and, for each $j\in\{1,...,\ell\}$, let
$$(\textbf g^{(j)}_k)_{k\in\N}=(g_{k,1}^{(j)},...,g_{k,\ell}^{(j)})_{k\in\N}$$
be a sequence in $G^\ell$. Suppose that 
 for any $t\in\{1,...,\ell\}$ and any  $j\in\{1,...,\ell\}$, $(T_{g_{k,t}^{(j)}})_{k\in\N}$ has the mixing property and that for any $t$ and any $i\neq j$, $(T_{(g_{k,t}^{(j)}-g_{k,t}^{(i)})})_{k\in\N}$ also has the mixing property. Then, there exists an infinite set $S\subseteq \N$ such that for any $A_0,...,A_\ell\in\mathcal A$,
$$\rlim{\alpha\in S^{(\ell)}}\mu(A_0\cap T_{g^{(1)}_\alpha}A_1\cap\cdots\cap T_{g^{(\ell)}_{\alpha}}A_\ell)=\prod_{j=0}^\ell \mu(A_j).$$
\end{prop}
\begin{proof}
The proof is by induction on $\ell$. When $\ell=1$, it follows from our hypothesis that for any $A_0,A_1\in\mathcal A$,
$$\rlim{\alpha\in\N^{(1)}}\mu(A_0\cap T_{g_\alpha^{(1)}}A_1)=\lim_{k\rightarrow\infty}\mu(A_0\cap T_{g_{k,1}^{(1)}}A_1)=\mu(A_0)\mu(A_1).$$
Now fix $\ell\in\N$ and suppose that  \cref{2.MainResult} holds for any $\ell'\leq \ell$.   
Let $\textbf X=(X,\mathcal A,\mu)$ and let $\mu_\Delta\in\mathcal C=\mathcal C(\underbrace{\textbf X,..., \textbf X}_{\ell+2\text{ times}})$ be defined by $\mu(A_0\times\cdots\times A_{\ell+1})=\mu(A_0\cap\cdots\cap A_{\ell+1})$. 
By the inductive hypothesis, there exists an infinite  $S\subseteq\N$  such that for any $A_1,...,A_{\ell+1}\in\mathcal A$,
    \begin{multline}\label{2.InductiveMixing}
        \rlim{\{j_1,...,j_\ell\}\in S^{(\ell)}}\mu_\Delta(X\times T_{g_{j_1,2}^{(1)}+\cdots+g_{j_\ell,\ell+1}^{(1)}}A_1\times\cdots\times T_{g_{j_1,2}^{(\ell+1)}+\cdots+g_{j_\ell,\ell+1}^{(\ell+1)}}A_{\ell+1})\\
        =\rlim{\{j_1,...,j_\ell\}\in S^{(\ell)}}\mu(X\cap T_{g_{j_1,2}^{(1)}+\cdots+g_{j_\ell,\ell+1}^{(1)}}A_1\cap\cdots\cap T_{g_{j_1,2}^{(\ell+1)}+\cdots+g_{j_\ell,\ell+1}^{(\ell+1)}}A_{\ell+1})\\
    =\rlim{\{j_1,...,j_\ell\}\in S^{(\ell)}}\mu( T_{g_{j_1,2}^{(1)}+\cdots+g_{j_\ell,\ell+1}^{(1)}}A_1\cap\cdots\cap T_{g_{j_1,2}^{(\ell+1)}+\cdots+g_{j_\ell,\ell+1}^{(\ell+1)}}A_{\ell+1})\\
    =\rlim{\{j_1,...,j_\ell\}\in S^{(\ell)}}\mu(A_1\cap T_{(g_{j_{1},2}^{(2)}-g_{j_{1},2}^{(1)})+\cdots+(g_{j_{\ell},\ell+1}^{(2)}-g_{j_{\ell},\ell+1}^{(1)})}A_2\cap
    \cdots\cap T_{(g_{j_1,2}^{(\ell+1)}-g_{j_{1},2}^{(1)})+\cdots+(g_{j_\ell,\ell+1}^{(\ell+1)}-g_{j_{\ell},\ell+1}^{(1)})}A_{\ell+1})\\
    =\prod_{j=1}^{\ell+1}\mu(A_j).
    \end{multline}
    By \cref{1.RBolzanoWierstrass} and the compactness of $\mathcal C$, there exists an infinite set $S_0\subseteq S$ and $\lambda_0\in\mathcal C$ such that for any $A_0,...,A_{\ell+1}\in\mathcal  A$,
    \begin{equation}\label{2.lambda0}
    \rlim{\{j_1,...,j_\ell\}\in S_0^{(\ell)}}\mu_\Delta(A_0\times T_{g_{j_1,2}^{(1)}+\cdots+g_{j_\ell,\ell+1}^{(1)}}A_1\times\cdots\times T_{g_{j_1,2}^{(\ell+1)}+\cdots+g_{j_\ell,\ell+1}^{(\ell+1)}}A_{\ell+1})=\lambda_0(\prod_{j=0}^{\ell+1}A_j).
    \end{equation}
    Likewise, there exist an infinite set $S_1\subseteq S_0$ and $\lambda\in\mathcal C$ such that for any $A_0,...,A_{\ell+1}\in\mathcal A$,
    \begin{multline}\label{2.lambda}
    \rlim{\{j_1,...,j_{\ell+1}\}\in S_1^{(\ell+1)}}\mu_\Delta(A_0\times T_{g_{j_1,1}^{(1)}+\cdots+g_{j_{\ell+1},\ell+1}^{(1)}}A_1\times\cdots\times T_{g_{j_1,1}^{(\ell+1)}+\cdots+g_{j_{\ell+1},\ell+1}^{(\ell+1)}}A_{\ell+1})=\lambda(\prod_{j=0}^{\ell+1}A_j).
    \end{multline}
Let $\textbf {Y}=(\prod_{j=1}^{\ell+1} X,\bigotimes_{j=1}^{\ell+1}\mathcal  A, 
\bigotimes_{j=1}^{\ell+1}\mu)$. Note that \eqref{2.InductiveMixing} holds if we substitute $S_1$ for $S$ and \eqref{2.lambda0} holds when we substitute $S_1$ for $S_0$. Performing this substitution and first applying \eqref{2.lambda0} and  then \eqref{2.InductiveMixing} to $A_1,...,A_{\ell+1}\in\mathcal A$, we have
$$\lambda_0(X\times A_1\times\cdots\times A_{\ell+1})=\prod_{j=1}^{\ell+1}\mu(A_{j}).$$
Also, trivially, for any $A_0\in\mathcal A$,
$$\lambda_0(A_0\times X\times\cdots\times X)=\mu(A_0).$$
Thus, $\lambda_0$ is  a coupling of $\textbf X$ and $\textbf Y$.\\
Using formula \eqref{2.lambda0}, \cref{2.DecomposingRlimLemma} and  applying \eqref{2.lambda} to the set $S_1=\{k_j\,|\,j\in\N\}$ (where we assume that $(k_j)_{j\in\N}$ is an increasing sequence), we have
\begin{multline}\label{2.LimitForProductsIsLambda}
    \lim_{t\rightarrow\infty}\lambda_0(A_0\times T_{g_{k_{t},1}^{(1)}}A_1\times\cdots\times T_{g_{k_{t},1}^{(\ell+1)}}A_{\ell+1} )\\
    =\lim_{t\rightarrow\infty} \rlim{\{j_2,...,j_{\ell+1}\}\in S_1^{(\ell)}}\mu_\Delta(A_0\times  T_{g_{j_2,2}^{(1)}+\cdots+g_{j_{\ell+1},\ell+1}^{(1)}}(T_{g_{k_{t},1}^{(1)}}A_1)\times\cdots\times T_{g_{j_2,2}^{(\ell+1)}+\cdots+g_{j_{\ell+1},\ell+1}^{(\ell+1)}}(T_{g_{k_{t},1}^{(\ell+1)}}A_{\ell+1}))\\
    =\lim_{t\rightarrow\infty} \rlim{\{j_2,...,j_{\ell+1}\}\in S_1^{(\ell)},\,k_t<j_2}\mu_\Delta(A_0\times  T_{g_{k_{t},1}^{(1)}+g_{j_2,2}^{(1)}+\cdots+g_{j_{\ell+1},\ell+1}^{(1)}}A_1\times\cdots\times T_{g_{k_{t},1}^{(\ell+1)}+g_{j_2,2}^{(\ell+1)}+\cdots+g_{j_{\ell+1},\ell+1}^{(\ell+1)}}A_{\ell+1})\\
     =\rlim{\{j_1,...,j_{\ell+1}\}\in S_1^{(\ell+1)}}\mu_\Delta(A_0\times  T_{g_{j_1,1}^{(1)}+\cdots+g_{j_{\ell+1},\ell+1}^{(1)}}A_1\times\cdots\times  T_{g_{j_1,1}^{(\ell+1)}+\cdots+g_{j_{\ell+1},\ell+1}^{(\ell+1)}}A_{\ell+1})=\lambda(\prod_{j=0}^{\ell+1}A_j),
\end{multline}
For each $j\in\N$, let $\textbf T_j=T_{g_{k_j,1}^{(1)}}\times\cdots\times T_{g_{k_j,1}^{(\ell+1)}}$. Note that for any increasing sequence $(t_s)_{s\in\N}$ in $\N$, there exists a subsequence $(t'_s)_{s\in\N}$ and a measure $\lambda'\in\mathcal C(\textbf X,\textbf Y)$, such that for any $A\in\mathcal A$ and any $B\in \bigotimes_{j=1}^{\ell+1}\mathcal  A$, $\lim_{s\rightarrow\infty}\lambda_0(A\times \textbf T_{t'_s}B)=\lambda'(A\times B)$. By \eqref{2.LimitForProductsIsLambda}, $\lambda'=\lambda$ and hence, for any $A\in\mathcal A$ and any $B\in \bigotimes_{j=1}^{\ell+1}\mathcal  A$, $\lim_{j\rightarrow\infty}\lambda_0(A\times \textbf T_j B)=\lambda(A\times B)$.\\
By \cref{2.PropTheIndependentjoining} applied to $\textbf X=(X,\mathcal A,\mu)$, $\textbf Y=(\prod_{j=1}^{\ell+1} X,\bigotimes_{j=1}^{\ell+1}\mathcal  A, 
\bigotimes_{j=1}^{\ell+1}\mu)$ and  the sequence of measure preserving transformations $(T^{-1}_{g_{k_j,1}^{(1)}}\times\cdots\times T^{-1}_{g_{k_j,1}^{(\ell+1)}})_{j\in\N}$, we have that $\lambda=\bigotimes_{j=0}^{\ell+1}\mu$. It follows that for any $A_0,...,A_{\ell+1}\in\mathcal A$,
\begin{multline*}
\rlim{\alpha\in S_1^{(\ell+1)}}\mu(A_0\cap T_{g^{(1)}_\alpha}A_1\cap\cdots\cap T_{g^{(\ell+1)}_{\alpha}}A_{\ell+1})\\
=\rlim{\alpha\in S_1^{(\ell+1)}}\mu_\Delta (A_0\times T_{g^{(1)}_\alpha}A_1\times\cdots\times T_{g^{(\ell+1)}_{\alpha}}A_{\ell+1})=\prod_{j=0}^{\ell+1} \mu(A_j),
\end{multline*}
completing the proof.
\end{proof}
\section{Strongly mixing systems are "almost" strongly mixing of  all orders}
In this section we will prove the following  theorem (\cref{1.MainResult} from the Introduction) which is the main result of this paper. 
\begin{thm}\label{3.MainResult}
Let $\ell\in\N$ and let $(X,\mathcal A,\mu, (T_g)_{g\in G})$ be a  measure preserving system. 
The following statements are equivalent:
\begin{enumerate}[(i)]
    \item  $(T_g)_{g\in G}$ is strongly mixing.
      \item For any $\ell$ non-degenerated and essentially distinct sequences 
    $$(\textbf g_k^{(j)})_{k\in\N}=(g^{(j)}_{k,1},...,g^{(j)}_{k,\ell})_{k\in\N},\text{ }j\in\{1,...,\ell\},$$
   in $G^{\ell}$, there exists an infinite $S\subseteq\N$ such that for any  $A_0,...,A_\ell\in\mathcal A$,
    \begin{equation}
        \rlim{\alpha\in S^{(\ell)}}\mu(A_0\cap T_{ g^{(1)}_\alpha}A_1\cap \cdots\cap T_{ g^{(\ell)}_\alpha}A_\ell)=\prod_{j=0}^\ell\mu(A_j).
    \end{equation}
    \item For any $\epsilon>0$ and any $A_0,...,A_\ell\in\mathcal A$, the set 
    $$R_\epsilon(A_0,...,A_\ell)=\{(g_1,...,g_\ell)\in G^\ell\,|\,|\mu(A_0\cap T_{g_1}A_1\cap\cdots \cap T_{g_\ell}A_\ell)-\prod_{j=0}^\ell \mu( A_j)|<\epsilon\}$$
    is $\tilde \Sigma_\ell^*$ in $G^\ell$.
    \item For any $\epsilon>0$ and any $A_0,A_1\in\mathcal A$, the set $R_\epsilon(A_0,A_1)$ is $\Sigma_\ell^*$ in $G$.
    \end{enumerate}
\end{thm}
\begin{proof}
(i)$\implies$(ii): Note that since $(T_g)_{g\in G}$ is strongly mixing,  
for any $t\in\{1,...,\ell\}$ and any  $j\in\{1,...,\ell\}$, $(T_{g_{k,t}^{(j)}})_{k\in\N}$ has the mixing property and that for any $t$ and any $i\neq j$, $(T_{(g_{k,t}^{(j)}-g_{k,t}^{(i)})})_{k\in\N}$ also has the mixing property. Thus (ii) follows from \cref{2.MainResult}.\\

(ii)$\implies$(iii): By (ii), we have that for any $\epsilon>0$, any $A_0,...,,A_\ell\in\mathcal A$ and any $\ell$ non-degenerated and essentially distinct sequences 
    $$(\textbf g_k^{(j)})_{k\in\N}=(g^{(j)}_{k,1},...,g^{(j)}_{k,\ell})_{k\in\N},\text{ }j\in\{1,...,\ell\},$$
   in $G^{\ell}$, there exists an $\alpha\in\N^{(\ell)}$ such that
   $$(g_\alpha^{(1)},...,g_\alpha^{(\ell)})\in R_\epsilon(A_0,...,A_\ell),$$
  which implies that $R_\epsilon(A_0,...,A_\ell)$ is $\tilde\Sigma_\ell^*$.\\
  
  (iii)$\implies$(iv): Let $\epsilon>0$, let $A_0,A_1\in\mathcal A$ and let 
  $(\textbf g^{(1)}_k)_{k\in\N}=(g_{k,1}^{(1)},...,g_{k,\ell}^{(\ell)})_{k\in\N}$ be a non-degenerated seequence in $G^\ell$. In order to prove that $\mathcal R_\epsilon(A_0,A_1)$ is $\Sigma_\ell^*$, it suffices to show that for some $\alpha\in\N^{(\ell)}$, $g_\alpha^{(1)}\in\mathcal R_\epsilon(A_0,A_1)$.\\
   Note that for any sequence  $(h^{(1)}_k)_{k\in\N}$ in $G$ with $\lim_{k\rightarrow\infty}h_k^{(1)}=\infty$ one can pick sequences $(h^{(2)}_k)_{k\in\N}$,...,
   $(h^{(\ell)}_k)_{k\in\N}$
   in $G$ with the property that  for any distinct $i,j\in\{1,...,\ell\}$, $$\lim_{k\rightarrow\infty}h^{(j)}_k=\infty\text{ and } \lim_{k\rightarrow\infty}(h^{(j)}_k-h^{(i)}_k)=\infty.$$  
   Hence, one can find non-degenerated sequences $(\textbf g_k^{(j)})_{k\in\N}$ in $G^\ell$, $j\in\{2,...,\ell\}$, such that $(\textbf g^{(1)}_k)_{k\in\N}$,...,
   $(\textbf g^{(\ell)}_k)_{k\in\N}$ are essentially distinct. By (iii), there exists an $\alpha\in \N^{(\ell)}$ for which 
   $$(g^{(1)}_\alpha,...,g^{(\ell)}_\alpha)\in\mathcal R_\epsilon(A_0,A_1,\underbrace{X,...,X}_{\ell-1\text{ times}}).$$ 
   This implies that $g^{(1)}_\alpha\in R_\epsilon(A_0,A_1)$.\\
  
  (iv)$\implies$(i):  Let $f\in L^2(\mu)$ be such that $\int_Xf\text{d}\mu=0$ and $\|f\|_{L^2}=1$. We will show that $\lim_{g\rightarrow\infty}T_gf=0$ in the weak topology of $L^2(\mu)$. To do this, it suffices to prove that for any sequence $(g_k)_{k\in\N}$ in $G$ with $\lim_{k\rightarrow\infty}g_k=\infty$, there exists an increasing sequence $(k_j)_{j\in\N}$ in $\N$  with $\lim_{j\rightarrow\infty}T_{g_{k_j}}f=0$. \\  
Note that 
$$\sigma(g)=\int_X\overline f T_g f\text{d}\mu,\,g\in G$$
is a positive definite function and hence, by Bochner's theorem, there is a unique Borel probability measure  $\rho$ on $\hat G$, the Pontryagin dual of $G$, with the property that for all $g\in G$,
\begin{equation}\label{3.DefnOfRo}
\int_X\overline fT_gf\text{d}\mu=\int_{\hat G} \phi_g(\chi)\text{d}\rho(\chi),
\end{equation}
where for each $\chi\in \hat G$ and each $g\in G$, $\phi_g(\chi)=\chi(g)$.\\
Let now $(g_k)_{k\in\N}$ be a sequence in $G$ with $\lim_{k\rightarrow\infty}g_k=\infty$. Let $(\textbf g_k)_{k\in\N}=(\underbrace{g_k,...,g_k}_{\ell\text{ times}})_{k\in\N}$ (note that $(\textbf g_k)_{k\in\N}$ is a non-degenerated sequence in $G^\ell$). We claim that there exists an increasing sequence   $(k_j)_{j\in\N}$ in $\N$ such that:
\begin{enumerate}
\item For some $K\in  L^2(\rho)$, 
\begin{equation}\label{3.AccumulationPoint}
K=\lim_{j\rightarrow\infty}\phi_{g_{k_j}}
\end{equation}
in the weak topology of $L^2(\rho)$.
\item Let $S=\{k_j\,|\,j\in\N\}$. There exists $H\in  L^2(\rho)$ such that 
\begin{equation}\label{3.AccumulationPointRLimit}
H=\rlim{\alpha\in S^{(\ell)}}\phi_{g_\alpha}=\lim_{j_1\rightarrow\infty}\cdots\lim_{j_\ell\rightarrow\infty}\phi_{(g_{k_{j_1}}+\cdots+g_{k_{j_\ell}})}
\end{equation}
in the weak topology of $L^2(\rho)$.
\item  For any $A_0,A_1\in\mathcal A$,  there exists a real number $r_{A_0,A_1}$ such that 
  \begin{equation}\label{3.PreMixingExpresionOnMainResult}
  \rlim{\alpha\in S^{(\ell)}}\mu(A_0\cap T_{-g_\alpha}A_1)=r_{A_0,A_1}.
  \end{equation}
\end{enumerate}
To establish the existence of such a sequence, one first invokes the   pre-compactness of the set $\{\phi_{g}\,|\,g\in G\}$ in the  weak topology of the set $L^2(\rho)$ to obtain an increasing  sequence $(k'_j)_{j\in\N}$ for which \eqref{3.AccumulationPoint} holds. Moreover, by using \cref{1.ItteratedLimitsRemark}, one can find a subsequence $(k_j'')_{j\in\N}$ of $(k_j')_{j\in\N}$ for which \eqref{3.AccumulationPointRLimit} holds for $S=\{k''_j\,|\,j\in\N\}$. Finally,  by a diagonalization argument, we can pick a subsequence $(k_j)_{j\in\N}$ of $(k''_j)_{j\in\N}$ for which \eqref{3.PreMixingExpresionOnMainResult} holds for  any $A_0,A_1$ from a countable dense subset of $\mathcal A$. If follows (by a standard approximation argument) that \eqref{3.PreMixingExpresionOnMainResult} holds for any $A_0,A_1\in\mathcal A$.\\

By (iv),  for every $A_0,A_1\in\mathcal A$, $r_{A_0,A_1}=\mu(A_0)\mu(A_1)$ (otherwise we would be able to find an $\epsilon>0$ for which the set $\mathcal R_\epsilon(A_0,A_1)$ is not $\Sigma_\ell^*$).
Since the linear combinations of indicator functions are dense in $L^2(\mu)$, it follows that  for any $f_1,f_2\in L^2(\mu)$,
\begin{equation}\label{3.GoingToProduct}
  \rlim{\alpha\in S^{(\ell)}}\int_X f_1 T_{g_\alpha}f_2\text{d}\mu=\int_X f_1\text{d}\mu\int_Xf_2\text{d}\mu.
\end{equation}
It follows from \eqref{3.DefnOfRo} and \eqref{3.GoingToProduct} that for any $g\in G$,
\begin{multline}\label{3.H=0}
\int_{\hat G} \overline{\phi_g(\chi)}H(\chi)\text{d}\rho(\chi)=\rlim{\alpha\in S^{(\ell)}}\int_{\hat G} \overline{\phi_{g}(\chi)}\phi_{g_\alpha}(\chi)\text{d}\rho(\chi)\\
=\rlim{\alpha\in S^{(\ell)}}\int_{\hat G} \phi_{-g}(\chi)\phi_{g_\alpha}(\chi)\text{d}\rho(\chi)=\rlim{\alpha\in S^{(\ell)}}\int_{\hat G} \phi_{(g_\alpha-g)}(\chi)\text{d}\rho(\chi)\\
=\rlim{\alpha\in S^{(\ell)}}\int_{X}\overline fT_{g_\alpha-g}f\text{d}\mu=\rlim{\alpha\in S^{(\ell)}}\int_{X}T_{g}\overline fT_{g_\alpha}f\text{d}\mu=\int_X\overline f\text{d}\mu\int_X f\text{d}\mu=0. 
\end{multline}
Since the linear combinations of the characters $\phi_g,\, g\in G$, are dense in $L^2(\rho)$, it follows from \eqref{3.H=0} that $H=0$. By \eqref{3.AccumulationPoint} and \eqref{3.AccumulationPointRLimit}, we have 
\begin{multline*}
0=H=\rlim{\alpha\in S^{(\ell)}}\phi_{g_\alpha}=\lim_{j_1\rightarrow\infty}\cdots\lim_{j_\ell \rightarrow\infty}\phi_{(g_{k_{j_1}}+\cdots+g_{k_{j_\ell}})}
=\lim_{j_1\rightarrow\infty}\cdots\lim_{j_{\ell}\rightarrow\infty}\prod_{t=1}^\ell \phi_{g_{k_{j_t}}}\\
=(\lim_{j_1\rightarrow\infty}\cdots\lim_{j_{\ell-1}\rightarrow\infty}\prod_{t=1}^{\ell-1} \phi_{g_{k_{j_t}}})(\lim_{j_\ell\rightarrow\infty}\phi_{g_{k_{j_\ell}}})=
(\lim_{j_1\rightarrow\infty}\cdots\lim_{j_{\ell-1}\rightarrow\infty}\prod_{t=1}^{\ell-1} \phi_{g_{k_{j_t}}})(\lim_{j\rightarrow\infty}\phi_{g_{k_{j}}})\\
=\cdots=\lim_{j_1\rightarrow\infty}\phi_{g_{k_{j_1}}}(\prod_{t=2}^{\ell}\lim_{j\rightarrow\infty} \phi_{g_{k_{j}}}) =\prod_{t=1}^\ell(\lim_{j\rightarrow\infty}\phi_{g_{k_j}})= K^{\ell}.
\end{multline*}
So, $K^\ell=0$  and hence $K=0$.\\
Consider now the closed and $(T_g)_{g\in G}$-invariant subspace $\mathcal H_f=\overline{\text{span}(\{T_g f\,|\,g\in G\})}\subseteq L^2(\mu)$. Since $K=0$, it follows from \eqref{3.DefnOfRo} and \eqref{3.AccumulationPoint} that for each $g\in G$,
$$\lim_{j\rightarrow\infty} \int_X T_g\overline fT_{g_{k_j}}f\text{d}\mu=\lim_{j\rightarrow\infty} \int_X \overline fT_{(g_{k_j}-g)}f\text{d}\mu=\lim_{j\rightarrow\infty}\int_{\hat G}\phi_{(g_{k_j}-g)}\text{d}\rho=\lim_{j\rightarrow\infty}\int_{\hat G}\overline{\phi_{g}}\phi_{g_{k_j}}\text{d}\rho=0.$$ 
It follows that for any $f'\in \mathcal H_f$, $\lim_{j\rightarrow\infty}\int_X\overline {f'}T_{g_{k_j}}f\text{d}\mu=0$. Noting that  $L^2(\mu)=\mathcal H_f\oplus \mathcal H_f^\perp$, we obtain that $\lim_{j\rightarrow\infty}T_{g_{k_j}}f=0$ in the weak topology of $L^2(\mu)$. In light of the remarks made at the begining of the proof of (iv)$\implies$(i), this, in turn, implies that $\lim_{g\rightarrow\infty}T_gf=0$. We are done. 
\end{proof}
\section{Some "diagonal" results for strongly mixing systems}
In order to give the reader the flavor of the main theme of this section, we start by formulating a slightly enhanced form of \cref{1.ZDiagonalResult} from the Introduction. (This theorem is a rather special case of the results of "diagonal" nature to be proved in this section.) 
\begin{prop}\label{4.ZExample}
Let $(X,\mathcal A,\mu, T)$ be a measure preserving system and let $a_1,...,a_\ell$  be non-zero distinct integers. Then $T$ is strongly mixing if and only if  for any $A_0,...,A_\ell\in\mathcal A$ and any $\epsilon>0$, the set 
$$\{n\in\Z\,|\,|\mu(A_0\cap T^{a_1n}A_1\cap\cdots\cap T^{a_\ell n}A_\ell)-\prod_{j=0}^\ell\mu(A_j)|<\epsilon\}$$
is $\Sigma_\ell^*$.
\end{prop}
We move now to formulations of more general "diagonal" results. \\

Let $(G,+)$ be a countable abelian group, let $(X,\mathcal A,\mu, (T_g)_{g\in G})$ be a measure preserving system, let $\ell\in\N$ and let $\phi_1,...,\phi_\ell:G\rightarrow G$ be homomorphisms. For any $\epsilon>0$ and any $A_0,...,A_\ell\in\mathcal A$, define
$$R_\epsilon^{\phi_1,...,\phi_\ell}(A_0,...,A_\ell)=\{g\in G\,|\,|\mu(A_0\cap T_{\phi_1(g)}A_1\cap\cdots\cap T_{\phi_\ell(g)}A_\ell)-\prod_{j=0}^\ell\mu(A_j)|<\epsilon\}.$$

We first give two equivalent formulations of a general result which deals with finitely generated groups.
\begin{thm}\label{4.FinitelyGeneratedEquivalence}
Let $(G,+)$ be a finitely generated abelian group, let $(X,\mathcal A,\mu, (T_g)_{g\in G})$ be a measure preserving system and let the homomorphisms $\phi_1,...,\phi_\ell:G\rightarrow G$ be such that for any $j\in\{1,...,\ell\}$, $\ker(\phi_j)$ is finite and for any $i\neq j$, $\ker(\phi_j-\phi_i)$ is also finite. Then $(T_g)_{g\in G}$ is strongly mixing if and only if for any $A_0,...,A_\ell\in\mathcal A$ and any $\epsilon>0$, the set $R_\epsilon^{\phi_1,...,\phi_\ell}(A_0,...,A_\ell)$ is $\Sigma_\ell^*$.
\end{thm}
Note that if $G$ is a finitely generated abelian group and $\phi:G\rightarrow G$ is a homomorphism, $\ker(\phi)$ is finite if and only if the index of $\phi(G)$ in $G$ is finite. It follows that \cref{4.FinitelyGeneratedEquivalence} can be formulated in the following equivalent form.
\begin{thm}\label{4.FinitelyGeneratedEquivalenceIndex}
Let $(G,+)$ be a finitely generated abelian group, let $(X,\mathcal A,\mu, (T_g)_{g\in G})$ be a measure preserving system and let the homomorphisms $\phi_1,...,\phi_\ell:G\rightarrow G$ be such that for any $j\in\{1,...,\ell\}$, the index of  $\phi_j(G)$ in $G$ is finite and for any $i\neq j$, the index  of $(\phi_j-\phi_i)$ in $G$ is also finite. Then $(T_g)_{g\in G}$ is strongly mixing if and only if for any $A_0,...,A_\ell\in\mathcal A$ and any $\epsilon>0$, the set $R_\epsilon^{\phi_1,...,\phi_\ell}(A_0,...,A_\ell)$ is $\Sigma_\ell^*$.
\end{thm}
We are going now to formulate and prove variants of Theorems \ref{4.FinitelyGeneratedEquivalence} and \ref{4.FinitelyGeneratedEquivalenceIndex} which pertain to mixing actions of general (not necessarily finitely generated) countable abelian groups. Unlike Theorems \ref{4.FinitelyGeneratedEquivalence} and \ref{4.FinitelyGeneratedEquivalenceIndex}, the following two  theorems are not equivalent. We will provide the relevant counterexamples at the end of this section.
\begin{thm}\label{4.InjectiveDiagonalResult}
Let $(G,+)$ be a countable abelian group, let $(X,\mathcal A,\mu,(T_g)_{g\in G})$ be a strongly mixing system and let the homomorphisms  $\phi_1,...,\phi_\ell:G\rightarrow G$ be such that for any $j\in\{1,...,\ell\}$, $\ker(\phi_j)$ is finite and for any $i\neq j$, $\ker(\phi_j-\phi_i)$ is also finite. For any non-degenerated sequence $(\textbf g_k)_{k\in\N}=(g_{k,1},...,g_{k,\ell})_{k\in\N}$ in $G^\ell$ there exists an infinite set $S\subseteq \N$ such that for any 
$A_0,...,A_\ell \in\mathcal A$,
$$\rlim{\alpha\in S^{(\ell)}}\mu(A_0\cap T_{\phi_1(g_\alpha)}A_1\cap\cdots\cap T_{\phi_\ell(g_\alpha)}A_\ell)=\prod_{j=0}^\ell \mu(A_j).$$
Equivalently, for any 
$A_0,...,A_\ell \in\mathcal A$ and any $\epsilon>0$, the set $R_\epsilon^{\phi_1,...,\phi_\ell}(A_0,...,A_\ell)$ is $\Sigma_\ell^*$. 
\end{thm}
\begin{proof}
Since for any distinct $i,j\in\{1,...,\ell\}$, $\ker(\phi_j)$ and $\ker(\phi_j-\phi_i)$ are both finite, we have  for each $t\in\{1,...,\ell\}$,
$$\lim_{k\rightarrow\infty}\phi_j(g_{k,t})=\infty\text{ and }\lim_{k\rightarrow\infty}(\phi_j(g_{k,t})-\phi_i(g_{k,t}))=\infty.$$
For each $j\in\{1,...,\ell\}$, let 
$$(\textbf g^{(j)}_k)_{k\in\N}=(\phi_j(g_{k,1}),...,\phi_j(g_{k,\ell}))_{k\in\N}.$$ 
Then the sequences $(\textbf g^{(1)}_k)_{k\in\N},...,(\textbf g^{(\ell)}_k)_{k\in\N}$ are non-degenerated and essentially distinct.  By \cref{3.MainResult}, (ii), there exists an infinite set $S\subseteq \N$ such that for any $A_0,...,A_\ell\in\mathcal A$,
\begin{multline*}
\rlim{\alpha\in S^{(\ell)}}\mu(A_0\cap T_{\phi_1(g_\alpha)}A_1\cap\cdots\cap T_{\phi_\ell(g_\alpha)}A_\ell)\\
=\rlim{\alpha\in S^{(\ell)}}\mu(A_0\cap T_{g^{(1)}_\alpha}A_1\cap\cdots\cap T_{g^{(\ell)}_\alpha}A_\ell)=\mu(\prod_{j=0}^\ell A_j).
\end{multline*}
\end{proof}
\begin{rem}\label{4.CombinatorialRemark}
The goal of this remark is to indicate an alternative way of proving \cref{4.InjectiveDiagonalResult}. Let $G$ and $\phi_1,...,\phi_\ell$ be as in the hypothesis of \cref{4.InjectiveDiagonalResult}. In Section 5 we will show that if $E$ is a $\tilde\Sigma_\ell^*$ set in $G^\ell$, then $\{g\in G\,|\,(\phi_1(g),...,\phi_\ell(g))\in E\}$ is a $\Sigma_\ell^*$ set in $G$ (see \cref{5.3.UsefulInSection4}). Thus, for any measure preserving system $(X,\mathcal A,\mu, (T_g)_{g\in G})$, any $A_0,....,A_\ell\in\mathcal A$ and any $\epsilon>0$, if $R_\epsilon(A_0,...,A_\ell)$ is a $\tilde\Sigma_\ell^*$ set, then $R_\epsilon^{\phi_1,...,\phi_\ell}(A_0,...,A_\ell)$ is a $\Sigma_\ell^*$ set. One can now invoke \cref{3.MainResult}, (iii).
\end{rem}
The next result complements \cref{4.InjectiveDiagonalResult}. Note that it provides a somewhat stronger version of one of the directions in \cref{4.FinitelyGeneratedEquivalenceIndex}.
\begin{thm}\label{4.quasiSurjectiveDiagonalResult}
Let $(G,+)$ be a countable abelian group, let $(X,\mathcal A,\mu,(T_g)_{g\in G})$ be a measure preserving system and let the homomorphisms $\phi_1,...,\phi_\ell:G\rightarrow G$ be such that either one of $\phi_1(G)$, $\phi_2(G)$ or $(\phi_2-\phi_1)(G)$ has finite index in $G$. If for all $A_0,...,A_\ell\in\mathcal A$ and all $\epsilon>0$ the set 
$R_\epsilon^{\phi_1,...,\phi_\ell}(A_0,...,A_\ell)$ is $\Sigma_\ell^*$, then $(T_g)_{g\in G}$ is strongly mixing. 
\end{thm}
\begin{proof}
We will assume that $(\phi_2-\phi_1)(G)$ has finite index in $G$, the other two cases can be handled similarly. For any $A_1,A_2\in\mathcal A$ and any $\epsilon>0$, we have
\begin{multline*}
R^{\phi_1,...,\phi_\ell}_\epsilon(X,A_1,A_2,\underbrace{X,...,X}_{\ell-2\text{ times}})\\
=\{g\in G\,|\,|\mu(X\cap T_{\phi_1(g)}A_1\cap T_{\phi_2(g)}A_2\cap T_{\phi_3(g)} X\cap\cdots\cap T_{\phi_\ell(g)}X)-\mu(A_1)\mu(A_2)|<\epsilon\}\\
=\{g\in G\,|\,|\mu(T_{\phi_1(g)}A_1\cap T_{\phi_2(g)}A_2)-\mu(A_1)\mu(A_2)|<\epsilon\}=R_\epsilon^{\phi_2-\phi_1}(A_1,A_2).
\end{multline*}
By our assumption, for any  $\epsilon>0$ and any $A_1,A_2\in\mathcal A$, the set $R^{\phi_2-\phi_1}_\epsilon(A_1,A_2)$ is a $\Sigma_\ell^*$ set and hence, by \cref{3.MainResult}, (iv), $(T_{(\phi_2-\phi_1)(g)})_{g\in G}$ is strongly mixing.\\

We will now prove that $(T_g)_{g\in G}$ is strongly mixing by showing that for any sequence 
$(g_k)_{k\in\N}$ in $G$ with $\lim_{k\rightarrow\infty}g_k=\infty$, there exists an increasing sequence $(k_j)_{j\in\N}$ in $\N$ with the property that for any $A_0,A_1\in\mathcal A$,
$$\lim_{j\rightarrow\infty}\mu(A_0\cap T_{g_{k_j}}A_1)=\mu(A_0)\mu(A_1).$$
Let $(g_k)_{k\in\N}$ be a sequence in $G$ with $\lim_{k\rightarrow \infty}g_k=\infty$. By assumption, $(\phi_2-\phi_1)(G)$ has finite index in $G$, so there  exists an increasing sequence $(k_j)_{j\in\N}$  in $\N$ and an element $\tau\in G$ for which $\{g_{k_j}+\tau\,|\,j\in\N\}\subseteq (\phi_2-\phi_1)(G)$. Since $(T_{(\phi_2-\phi_1)(g)})_{g\in G}$ is strongly mixing, for any $A_0,A_1\in\mathcal A$,
$$\lim_{j\rightarrow\infty}\mu(A_0\cap T_{g_{k_j}}A_1)=\lim_{j\rightarrow\infty}\mu(A_0\cap T_{g_{k_j}+\tau}(T_{-\tau}A_1))=\mu(A_0)\mu(A_1),$$
completing the proof.
\end{proof}

The following proposition shows that the assumption made in  \cref{4.FinitelyGeneratedEquivalence} that $G$ is finitely generated cannot be removed. 
\begin{prop}\label{4.NonFiniteIndexExample}
Let $G=\bigoplus_{k\in\N} \Z$ and let $\ell\in\N$. There exists  a measure preserving system $(X,\mathcal A,\mu,(T_g)_{g\in G})$ and homomorphisms $\phi_1,...,\phi_\ell:G\rightarrow G$ satisfyng (a) for any $j\in\{1,...,\ell\}$, $\ker(\phi_j)$ is finite, and (b) for any $i\neq j$, $\ker(\phi_j-\phi_i)$ is also finite, and such that every set of the form $R_\epsilon^{\phi_1,...,\phi_\ell} (A_0,...,A_\ell)$ is $\Sigma_\ell^*$ but $(T_g)_{g\in G}$ is not strongly mixing.
\end{prop}
\begin{proof}
We will only carry out the proof for $\ell=2$, the general case can be handled similarly. Let $\phi_1:G\rightarrow G$ be the homomorphism given by
$$\phi_1((a_1,a_2,...,a_n,...))=(0,a_1,0,a_2,...,0,a_n,...).$$
Note that $\phi_1$ is injective (and so, $\ker(\phi_1)$ is trivial).\\
Let $X=\{0,1\}^G$ be endowed with the product topology, let $\mu$ be the $(\frac{1}{2},\frac{1}{2})$ product measure on $\mathcal A=\text{Borel}(X)$ and for each $g\in G$, let $S_g:X\rightarrow X$  be the map defined by $(S_g(x))(h)=x(h+g)$. The system $(X,\mathcal A,\mu,(S_g)_{g\in G})$ is strongly mixing. Define a measure preserving $G$-action  $(T_g)_{g\in G}$ on $(X,\mathcal A,\mu)$ by 
$$T_{(a_1,a_2,...)}=S_{(a_2,a_4,...)}$$
and let $\phi_2:G\rightarrow G$ be defined by $\phi_2(g)=2\phi_1(g)$. Note that for any $g=(a_1,a_2,...)\in G$,
$$T_{\phi_1(g)}=T_{\phi_1((a_1,a_2,...))}=T_{(0,a_1,0,a_2,...)}=S_{(a_1,a_2,...)}=S_g.$$
So, for any $\epsilon>0$ and any $A_0,A_1,A_2\in\mathcal A$,
\begin{multline}\label{4.SetWithTSetWithS}
R_\epsilon^{\phi_1,\phi_2}(A_0,A_1,A_2)\\
=\{g\in G\,|\,|\mu(A_0\cap T_{\phi_1(g)}A_1\cap T_{\phi_2(g)}A_2)-\mu(A_0)\mu(A_1)\mu(A_2)|<\epsilon\}\\
=\{g\in G\,|\,|\mu(A_0\cap S_{g}A_1\cap S_{2g}A_2)-\mu(A_0)\mu(A_1)\mu(A_2)|<\epsilon\}.
\end{multline}
It follows from \cref{4.InjectiveDiagonalResult} that every set of the form 
$$\{g\in G\,|\,|\mu(A_0\cap S_{g}A_1\cap S_{2g}A_2)-\mu(A_0)\mu(A_1)\mu(A_2)|<\epsilon\}$$
is $\Sigma_2^*$ and hence, by \eqref{4.SetWithTSetWithS}, for any any $A_0,A_1,A_2$ and any  $\epsilon>0$,  
$R_\epsilon^{\phi_1,\phi_2}(A_0,A_1,A_2)$ is  $\Sigma_2^*$.\\
Noting that for each $k\in\N$, $T_{(k,0,0,...)}=S_{(0,0,...)}$ is the identity map on $X$, we see that $(T_g)_{g\in G}$ is not strongly mixing. We are done.
\end{proof}
The next result shows that \cref{4.FinitelyGeneratedEquivalenceIndex} cannot be extended to arbitrary countable abelian groups.
\begin{prop}\label{4.LackOfInjectivityExample}
Let $G=\bigoplus_{k\in\N} \Z$ and let $\ell\in\N$. There exist a strongly mixing system $(X,\mathcal A,\mu, (T_g)_{g\in G})$ and homomorphisms $\phi_1,...,\phi_\ell:G\rightarrow G$ satisfying (a) for any $j\in\{1,...,\ell\}$, $\phi_j(G)=G$, and (b) for any $i\neq j$, $(\phi_i-\phi_j)(G)=G$, and
such that for some $A\in\mathcal A$ and some $\epsilon>0$, the set $R_\epsilon^{\phi_1,...,\phi_\ell}(A,...,A)$ is not $\Sigma_\ell^*$.
\end{prop}
\begin{proof}
Let $(X,\mathcal A,\mu, (T_g)_{g\in G})$ be a strongly mixing system and let $p_1,...,p_\ell\in\N$ be  $\ell$ different prime numbers. For each $j\in \{1,...,\ell\}$, let $\phi_j:G\rightarrow G$ be defined by
$$\phi_j(a_1,a_2,a_3,...)=(a_{p_j^1},a_{p_j^2},a_{p_j^3}...).$$
It follows that for any $j\in\{1,...,\ell\}$, $\phi_j(G)=G$ and since for any distinct $i,j\in\{1,...,\ell\}$ the sets $\{p_i^k\,|\,k\in\N\}$ and $\{p_j^k\,|\,k\in\N\}$ are disjoint, we have that $(\phi_j-\phi_i)(G)=G$ as well.\\
Observe that  the subgroup $G'=\{(a_1,0,0,...)\in G\,|\,a_1\in\Z\}$ is isomorphic to $\Z$ and that for any $j\in\{1,...,\ell\}$, $G'\subseteq\ker(\phi_j)$. Let $(g_k)_{k\in\N}$ be a sequence in $G'$ with $\lim_{k\rightarrow\infty}g_k=\infty$. Since for each $k\in\N$, $T_{\phi_j(g_k)}=T_{(0,0,...)}=\text{Id}$, where $\text{Id}$ is the identity map on $X$, we  have that for any $A\in \mathcal A$ with $\mu(A)\in(0,1)$, and any $k_1<\cdots<k_\ell$, 
$$\mu(A\cap T_{\phi_1(g_{k_1}+\cdots+g_{k_\ell})}A\cap\cdots\cap T_{\phi_\ell(g_{k_1}+\cdots+g_{k_\ell})}A)=\mu(A)\neq \mu^{\ell+1}(A).$$
It  follows that if $\epsilon$ is small enough,  the set $R_\epsilon^{\phi_1,...,\phi_\ell}(A,...,A)$ does not intersect the $\Sigma_\ell$ set 
$$\{g_{k_1}+\cdots+g_{k_\ell}\,|\,k_1<\cdots<k_\ell\}$$
and hence, it is not $\Sigma_\ell^*$. This completes the proof.
\end{proof}
\section{Largeness properties of  $\tilde\Sigma_m^*$ sets}
As we have seen above, any strongly mixing system $(X,\mathcal A,\mu,(T_g)_{g\in G})$ has the property that the sets $R_\epsilon(A_0,...,A_m)$ are $\tilde\Sigma_m^*$  (moreover, the strong mixing of $(T_g)_{g\in G}$ is characterized by this property). This section is devoted to the discussion of massivity and ubiquity of $\tilde\Sigma_m^*$ sets. 
Since strong mixing is a stronger property than those of mild  and weak mixing, one should expect that the notions of largeness associated with (multiple) mild and weak mixing are "majorized" by the notion of largeness associated with $\tilde\Sigma_m^*$ sets.
This will be established in Subsections 5.1 and 5.2. Finally, in Subsection 5.3 we will show that $\tilde\Sigma_m^*$ sets are ubiquitous in the sense that they are well spread among the cosets of \textit{admissible} subgroups of $G^m$ (the class of admissible subgroups will be introduced in Subsection 5.3). 
\subsection{Any $\tilde\Sigma_m^*$ set in $G^d$  is an \text{\rm{$\tilde{\text{IP}}\rm{^*}$}} set} 
In this section we will introduce \text{\rm{$\tilde{\text{IP}}\rm{^*}$}} sets and  juxtapose them with $\tilde\Sigma_m^*$ sets.
 (\text{\rm{$\tilde{\text{IP}}\rm{^*}$}} sets are intrinsically linked to the multiple mixing properties of mildly mixing systems. The connection between  \text{\rm{$\tilde{\text{IP}}\rm{^*}$}} sets and mildly mixing systems will be addressed in Section 6.)\\

Let $(G,+)$ be a countable abelian group and let $\mathcal F$ denote the set of  all  non-empty finite subsets of $\N$. Given a sequence $(g_k)_{k\in\N}$ in $G$, define an  $\mathcal F$-sequence $(g_\alpha)_{\alpha\in\mathcal F}$ by  
\begin{equation}\label{5.1.DefnAlphaIP}
g_\alpha=\sum_{j\in\alpha} g_j=g_{k_1}+\cdots+g_{k_t},\,\alpha=\{k_1,...,k_t\}.
\end{equation}
We will write
$$\lim_{\alpha\rightarrow\infty}g_\alpha=\infty$$
if for every finite $K\subseteq G$, there exists an $\alpha_0\in\mathcal F$ such that for any   $\alpha\in\mathcal F$ with $\alpha>\alpha_0$ (i.e. $\min \alpha>\max \alpha_0$), $g_\alpha\not\in K$.\\
A set $E\subseteq G$ is called an IP set if $E=\{g_\alpha\,|\,\alpha\in\mathcal F\}$ for some  sequence $(g_k)_{k\in\N}$ in $G$ such that  $\lim_{\alpha\rightarrow\infty}g_\alpha=\infty$.\footnote{
 IP sets are often defined  just as sets of the form 
$$\text{FS}((g_k)_{k\in\N})=\{g_{k_1}+\cdots+g_{k_t}\,|\,k_1<\cdots<k_t,\,t\in\N\}=\{g_\alpha\,|\,\alpha\in\mathcal F\}$$
(without the requirement that $\lim_{\alpha\rightarrow\infty}g_\alpha=\infty$). Our choice of definition for IP sets is dictated by our interest in the study of asymptotic properties of measure preserving actions. The distinction between our definition and the more traditional one is rather mild: for any infinite set of the form $E=\{g_\alpha\,|\,\alpha\in\mathcal F\}$ there exists a sequence $(h_k)_{k\in\N}$ such that $\{h_\alpha\,|\,\alpha\in\mathcal F\}\subseteq E$ and  $\lim_{\alpha\rightarrow\infty}h_\alpha=\infty$.  
} 
A set $E\subseteq G$ is called IP$^*$ if it has  a non-trivial intersection with every IP set. \\ 

We now introduce modifications of IP and IP$^*$ sets, namely $\tilde{\text{IP}}$ sets and \text{\rm{$\tilde{\text{IP}}\rm{^*}$}}  sets, which, as will be seen in Section 6, are naturally linked with  the properties of the sets  $R_\epsilon(A_0,...,A_\ell\}$ in the context of mildly mixing systems.
\begin{defn}\label{5.1.DefnOfIPtilde}
Let $(G,+)$ be a countable abelian group and let $d\in \N$. We say that  a set $E\subseteq G^d$ is an $\tilde{\text{IP}}$ set if it is of the form 
$$E=\{(g_\alpha^{(1)},...,g_\alpha^{(d)})\,|\,\alpha\in\mathcal F\},$$
where for each $j\in\{1,...,d\}$, $\{g_\alpha^{(j)}\,|\,\alpha\in\mathcal F\}$ is generated by $(g_{k}^{(j)})_{k\in\N}$ as in \eqref{5.1.DefnAlphaIP} and, in addition, for any $j\in\{1,...,d\}$,
\begin{equation}\label{5.1.IPcondition1}
\lim_{\alpha\rightarrow\infty}g^{(j)}_\alpha=\infty
\end{equation}
and for any $i\neq j$,
\begin{equation}\label{5.1.IPcondition2}
\lim_{\alpha\rightarrow\infty}(g^{(j)}_\alpha-g_\alpha^{(i)})=\infty.
\end{equation}
(Note that if $d=1$, then $E\subseteq G$ is an \text{IP} set if and only if it is an \text{\rm{$\tilde{\text{IP}}$}} set.)\\ 
A set $E\subseteq G^d$ is called an \text{\rm{$\tilde{\text{IP}}\rm{^*}$}} set if it has a non-trivial intersection with every $\tilde{\text{IP}}$ set in $G^d$. 
\end{defn}
\begin{rem}\label{5.1.CommonSenseSequence}
Let $(G,+)$ be a countable abelian group, let $d\in \N$ and let $E\subseteq G^d$ be an $\tilde{\text{IP}}$ set. From now on, whenever we pick a sequence $(\textbf g_k)_{k\in\N}=(g_k^{(1)},...,g_k^{(d)})_{k\in\N}$ in $G^d$ with the property that $E=\{(g_\alpha^{(1)},...,g_\alpha^{(d)})\,|\,\alpha\in\mathcal F\}$, we will tacitly assume that $(g_k^{(1)})_{k\in\N}$,...,$(g_k^{(d)})_{k\in\N}$ satisfy \eqref{5.1.IPcondition1} and \eqref{5.1.IPcondition2}.
\end{rem}

The following lemma unveils an important connection between {\rm{$\tilde{\text{IP}}$}} and $\tilde\Sigma_m$ sets.
\begin{lem}\label{5.1.SigmaInEveryIP}
Let $(G,+)$ be a countable abelian group and let $d,m\in\N$. Any {\rm{$\tilde{\text{IP}}$}} set $E\subseteq G^d$ contains a $\tilde\Sigma_m$ set. Namely, there exist non-degenerated and essentially distinct sequences  
$$(\textbf g_k^{(j)})=(g^{(j)}_{k,1},...,g^{(j)}_{k,m})_{k\in\N},\,j\in\{1,...,d\}$$
in $G^m$ with the property that $\{(g_\alpha^{(1)},...,g_\alpha^{(d)})\,|\,\alpha\in\N^{(m)}\}\subseteq E$, where for each $j\in\{1,...,d\}$ and each $\alpha=\{k_1,...,k_m\}\in\N^{(m)}$, $g_\alpha^{(j)}=g_{k_1,1}^{(j)}+\cdots+g_{k_m,m}^{(j)}$.

\end{lem}
\begin{proof}
Let  $E$ be an $\tilde{\text{IP}}$ set and let $(\textbf h_k)_{k\in\N}= (h_k^{(1)},...,h_k^{(d)})_{k\in\N}$ be such that 
$$E=\{\textbf h_\alpha\,|\,\alpha\in\mathcal F\}=\{(h_\alpha^{(1)},...,h_\alpha^{(d)})\,|\,\alpha\in\mathcal F\}.$$
Following the stipulation made in  \Cref{5.1.CommonSenseSequence}, for any finite set $F\subseteq G$, we can find an $\alpha_F\in\mathcal F$ such that for any $\alpha\in\mathcal F$ with $\alpha> \alpha_F$ and any distinct $i,j\in\{1,...d\}$, $h^{(j)}_{\alpha}\not\in F$ and $(h^{(j)}_{\alpha}-h^{(i)}_{\alpha})\not\in F$. In particular,  for any distinct $i,j\in\{1,...,d\}$
\begin{equation}\label{5.1.AlphaSequenceGoingToInfty} 
\lim_{k\rightarrow\infty}h_{k}^{(j)}=\infty\text{ and }\lim_{k\rightarrow\infty}(h_{k}^{(j)}-h_{k}^{(i)})=\infty. 
\end{equation}
For each $j\in\{1,...,d\}$ and  each $k\in\N$ we let 
\begin{equation}\label{5.1.DefnSequenceg_k}
\textbf g_k^{(j)}=\underbrace{(h_{k}^{(j)},...,h_{k}^{(j)})}_{m\text{ times}}.
\end{equation}

Note that by \eqref{5.1.AlphaSequenceGoingToInfty}, the sequences $(\textbf g^{(1)}_k)_{k\in\N}$,...,$(\textbf g^{(d)}_k)_{k\in\N}$ are non-degenerated and essentially distinct. It follows now from  \eqref{5.1.DefnSequenceg_k} that for any $\alpha=\{k_1,...,k_m\}\in\N^{(m)}$,
$$(g^{(1)}_\alpha,...,g^{(d)}_\alpha)=(\sum_{j=1}^m h_{k_j}^{(1)},...,\sum_{j=1}^m h_{k_j}^{(d)})=(h_{\{k_1,...,k_m\}}^{(1)},...,h_{\{k_1,...,k_m\}}^{(d)})\in E,$$
which completes the proof.
\end{proof}
\begin{rem}
The proof of \cref{5.1.SigmaInEveryIP} actually shows that any  {\rm{$\tilde{\text{IP}}$}} set is a union of $\tilde\Sigma_t$ sets. Let  $E\subseteq G^d$ be an {\rm{$\tilde{\text{IP}}$}} set and let $(\textbf g_k)_{k\in\N}$ be a sequence such that $E=\{\textbf g_\alpha\,|\alpha\in\mathcal F\}$. The proof of \cref{5.1.SigmaInEveryIP} shows that for each $t\in\N$, $\{\textbf g_{k_1}+\cdots+\textbf g_{k_t}\,|\,k_1<\cdots<k_t\}$  is a $\tilde\Sigma_t$ set. Hence,
$$E=\bigcup_{t\in\N}\{\textbf g_{k_1}+\cdots+\textbf g_{k_t}\,|\,k_1<\cdots<k_t\}.$$
\end{rem}
As an immediate consequence of \cref{5.1.SigmaInEveryIP} we have the following result.
\begin{cor}\label{5.1.EverySigma*IsIP*}
Let $(G,+)$ be a countable abelian group and let $d,m\in\N$. Every $\tilde \Sigma_m^*$ set in $G^d$ is an \text{\rm{$\tilde{\text{IP}}\rm{^*}$}} set.
\end{cor}
\begin{proof}
Let $E\subseteq G^d$ be a $\tilde\Sigma_m^*$ set and let $D\subseteq G^d$ be an  \rm{$\tilde{\text{IP}}$} set. By \cref{5.1.SigmaInEveryIP}, we have that $D$ contains  a $\tilde\Sigma_m$ set and hence $E\cap D\neq \emptyset$. Since $D$ was arbitrary, this shows that $E$ is an \rm{$\tilde{\text{IP}}\rm{^*}$} set.
\end{proof}
\subsection{Any $\tilde\Sigma_m^*$ set in $G^d$  has uniform density one}
We start with defining the notions of \textit{upper density} and \textit{uniform density one} in countable abelian groups. 
\begin{defn}
Let $(G,+)$ be a countable abelian group, let $E\subseteq G$ and let $(F_k)_{k\in\N}$ be a F{\o}lner sequence in $G$.\footnote{
Let $(G,+)$ be a countable abelian group. A sequence $(F_k)_{k\in\N}$ of non-empty finite subsets of $G$  is a F{\o}lner sequence if for any $g\in G$,
$$\lim_{k\rightarrow\infty}\frac{|(g+F_k)\cap F_k|}{|F_k|}=1,$$
where, for a finite set $A$, $|A|$ denotes its cardinality. It is well known that every countable abelian group contains a F{\o}lner sequence. 
} 
The \textbf{upper density} of $E$ with respect to $(F_k)_{k\in\N}$ is defined by 
$$\overline d_{(F_k)}(E)=\limsup_{k\rightarrow\infty}\frac{|E\cap F_k|}{|F_k|}.$$
A set $E\subseteq G$ has \textbf{uniform density one} if for every F{\o}lner sequence $(F_k)_{k\in\N}$, $\overline d_{(F_k)}(E)=1$.
\end{defn}
Sets of uniform density one are intrinsically connected with weakly mixing measure preserving systems. Recall that a  measure preserving action $(T_g)_{g\in G}$ on a probability space $(X,\mathcal A,\mu)$ is called weakly mixing if the diagonal action $(T_g\times T_g)_{g\in G}$ on $X\times X$ is ergodic. When $G$ is an amenable group, the notion of weak mixing can be equivalently defined with the help of strong C{\'e}saro limits along F{\o}lner sequences. Namely, $(T_g)_{g\in G}$ is weakly mixing if and only if for any F{\o}lner sequence $(F_k)_{k\in\N}$ and any $A_0,A_1\in\mathcal A$,  
$$\lim_{k\rightarrow\infty}\frac{1}{|F_k|}\sum_{g\in F_k}|\mu(A_0\cap T_gA_1)-\mu(A_0)\mu(A_1)|=0.$$
It  follows that $(T_g)_{g\in G}$ is weakly mixing if and only if the sets
$$R_\epsilon(A_0,A_1)=\{g\in G\,|\,|\mu(A_0\cap T_gA_1)-\mu(A_0)\mu(A_1)|<\epsilon\}$$
have uniform density one. The reader will find a few more equivalent forms of weak mixing  in \cref{6.2.EquivalentFormsOFWM} below. \\

In order to derive the  main result of this subsection, namely the fact that every $\tilde\Sigma_m^*$ set has uniform density one, we need first to prove two auxiliary propositions. 
\begin{prop} \label{5.2.IPPoincare}
Let $(G,+)$ be a countable abelian group, let $d\in\N$ and let $(F_k)_{k\in\N}$ be a F{\o}lner sequence in $G^d$. For any $E\subseteq G^d$ with $\overline d_{(F_k)}(E)>0$ and any \text{\rm{$\tilde{\text{IP}}$}} set $D\subseteq G^d$, there exists a sequence $(\textbf g_k)_{k\in\N}=(g_k^{(1)},...,g_k^{(d)})$ in $G^d$ such that (a) $\{\textbf g_\alpha\,|\,\alpha\in\mathcal F\}\subseteq D$,  (b) for any distinct $i,j\in\{1,...,d\}$, \eqref{5.1.IPcondition1} and \eqref{5.1.IPcondition2} hold, and (c) for any $\alpha\in \mathcal F$,
\begin{equation}\label{5.2.IPPoincareMultirecurrence}
\overline d_{(F_k)}(\bigcap_{\beta\subseteq\alpha,\,\beta\neq\emptyset}(E-\textbf g_\beta))>0.
\end{equation}
In other words, for each $\alpha\in \mathcal F$, the set $E_\alpha=\{\textbf h\in G^d\,|\,\forall \beta\subseteq \alpha,\,\beta\neq\emptyset,\,\textbf h+\textbf g_\beta\in E\}$ satisfies $\overline d_{(F_k)}(E_\alpha)>0$.
\end{prop}
\begin{proof}
Let $D=\{\textbf h_\alpha\,|\,\alpha\in\mathcal F\}$ be an \text{\rm{$\tilde{\text{IP}}$}} set in $G^d$ generated by the sequence $(\textbf h_k)_{k\in\N}=(h_{k,1},...,h_{k,d})_{k\in\N}$. 
We claim that for any $M\in\N$ with  $M>\frac{1}{\overline d_{(F_k)}(E)}$, there exist $L,R\in\N$, $L<R\leq M$ for which 
$\overline d_{(F_k)}(E\cap (E-\textbf h_{\{L+1,L+2,...,R\}}))>0$. To see this, suppose for the sake of contradiction that for any distinct $R,L\in\{1,...,M\}$, $R>L$, $\overline d_{(F_k)}(E\cap(E-\textbf h_{\{L+1,...,R\}}))=0$. Since  $\overline d_{(F_k)}$  is translation invariant and for any $L,R\in\{1,...,M\}$, $L<R$, $\textbf h_{\{L+1,...,R\}}=\textbf h_{\{1,...,R\}}-\textbf h_{\{1,...,L\}}$, we have that 
$$\overline d_{(F_k)}(E\cap(E-\textbf h_{\{L+1,...,R\}}))=\overline d_{(F_k)}((E-\textbf h_{\{1,...,L\}})\cap(E-\textbf h_{\{1,...,R\}}))=0.$$ 
It follows that
$$\overline d_{(F_k)}(\bigcup_{R=1}^M (E-\textbf h_{\{1,...,R\}}))=\sum_{R=1}^M\overline d_{(F_k)}(E-\textbf h_{\{1,...,R\}})=M\overline d_{(F_k)}(E)>1,$$ 
a contradiction. Thus, there exist $L,R\in\N$ with $L<R\leq M$ such that $\overline d_{(F_k)}(E\cap (E-\textbf h_{\{L+1,...,R\}}))>0$. We will let $\gamma_1=\{L+1,...,,R\}$.\\
Now let $E_1=E\cap (E-\textbf h_{\gamma_1})$. Repeating the above argument, we  find
$L',R'\in\N$, $R<L'<R'$, such that $\gamma_2=\{L'+1,...,R'\}$ satisfies
 $\overline d_{(F_k)}(E_1\cap (E_1-\textbf h_{\gamma_2}))>0$. It follows that $\gamma_1<\gamma_2$ and  that $\textbf h_{\gamma_1\cup \gamma_2}=\textbf h_{\gamma_1}+\textbf h_{\gamma_2}$. Hence
$$\overline d_{(F_k)}(E\cap(E-\textbf h_{\gamma_1})\cap(E-\textbf h_{\gamma_2})\cap (E-\textbf h_{\gamma_1\cup \gamma_2})>0.$$
Continuing in this way, we can find a sequence $(\gamma_k)_{k\in\N}$ with $\gamma_k<\gamma_{k+1}$ for each $k\in\N$ and the property that for any $\alpha\in \mathcal F$, 
$$\overline d_{(F_k)}(\bigcap_{\beta\subseteq\alpha,\,\beta\neq\emptyset}(E-\textbf h_{\bigcup_{k\in\beta}\gamma_k}))>0.$$

For each $k\in\N$, let $\textbf g_k=\textbf h_{\gamma_k}$ and for each $\alpha\in\mathcal F$, let $\textbf g_\alpha=\sum_{j\in\alpha}\textbf g_j=\textbf h_{\bigcup_{j\in\alpha}\gamma_j}$. Observe that the sequence $(\textbf g_\alpha)_{\alpha\in\mathcal F}$ satisfies \eqref{5.2.IPPoincareMultirecurrence}. Let $D'=\{\textbf g_\alpha\,|\,\alpha\in\mathcal F\}$. Clearly   $D'\subseteq D$. To finish the proof observe that 
$$(\textbf g_\alpha)_{\alpha\in\mathcal F}=( g_{\alpha,1},...,g_{\alpha,d})_{\alpha\in\mathcal F}=(h_{(\bigcup_{k\in\alpha}\gamma_k),1},...,h_{(\bigcup_{k\in\alpha}\gamma_k),d})_{\alpha\in\mathcal F}$$
 satisfies \eqref{5.1.IPcondition1} and \eqref{5.1.IPcondition2}.  Indeed, in view of  \Cref{5.1.CommonSenseSequence}, for any  $j\in\{1,...,d\}$, 
$$\lim_{\alpha\rightarrow\infty} g_{\alpha,j}=\lim_{\alpha\rightarrow\infty}h_{(\bigcup_{k\in\alpha}\gamma_k),j}=\infty
$$
and for $i\neq j$,
$$\lim_{\alpha\rightarrow\infty}( g_{\alpha,j}- g_{\alpha,i})=\lim_{\alpha\rightarrow\infty}(h_{(\bigcup_{k\in\alpha}\gamma_k),j}-h_{(\bigcup_{k\in\alpha}\gamma_k),i})=\infty.$$
\end{proof}
\begin{prop}\label{5.2.FiniteSumsInPositiveDensitySets}
Let $(G,+)$ be a countable abelian group, let $d,m\in\N$ and let $(F_k)_{k\in\N}$ be a F{\o}lner sequence in $G^d$. Any $E\subseteq G^d$ with $\overline d_{(F_k)}(E)>0$ contains a $\tilde\Sigma_m$ set. Namely, there exist non-degenerated and essentially distinct sequences  
$$(\textbf g_k^{(j)})=(g^{(j)}_{k,1},...,g^{(j)}_{k,m})_{k\in\N},\,j\in\{1,...,d\}$$
in $G^m$ with the property that $\{(g_\alpha^{(1)},...,g_\alpha^{(d)})\,|\,\alpha\in\N^{(m)}\}\subseteq E$.
\end{prop}
\begin{proof}
Fix $d\in\N$ and let $D$ be an \text{\rm{$\tilde{\text{IP}}$}} set  in $G^d$. Let  $(\textbf h_k)_{k\in\N}=(h_k^{(1)},...,h_k^{(d)})_{k\in\N}$ be a sequence in $G^d$ with $D=\{\textbf h_\alpha\,|\,\alpha\in\mathcal F\}$. Invoking \cref{5.2.IPPoincare} and  passing, if needed, to  a sub-\text{\rm{$\tilde{\text{IP}}$}} set in $D$, we can assume that for any $\alpha\in\mathcal F$,
\begin{equation}\label{5.2.IPReccurence}
\overline d_{(F_k)}(\bigcap_{\beta\subseteq\alpha,\,\beta\neq\emptyset}(E-\textbf h_\beta))>0
\end{equation}
and that $(\textbf h_k)_{k\in\N}$ satisfies \eqref{5.1.IPcondition1} and \eqref{5.1.IPcondition2}.\\
Let $m=1$. There exists a sequence $(\alpha_k)_{k\in\N}$ in $\mathcal F$ such that for each $k\in\N$, $ \alpha_k<\alpha_{k+1}$  and such that  for any distinct $k,k'\in\N$ and any distinct $i,j\in\{1,...,d\}$, 
\begin{equation}\label{5.2.Disjoint}
h_{\alpha_k}^{(j)}\neq h_{\alpha_{k'}}^{(j)}\text{ and }h_{\alpha_k}^{(j)}-h_{\alpha_k}^{(i)}\neq h_{\alpha_{k'}}^{(j)}-h_{\alpha_{k'}}^{(i)}.
\end{equation}
Pick a sequence $(A_k)_{k\in\N}$ of finite subsets of $G$ with the properties that for each $k\in\N$, (a) $|A_k|=k$, (b) $A_k\subseteq A_{k+1}$,  and (c)  $\bigcup_{k\in\N} A_k=G$.  By \eqref{5.2.IPReccurence}, for each $k\in\N$ we can find  $\textbf b_k=(b_{k,1},...,b_{k,d})$ in $G^d$ such that for any $t\in\{1,...,kd^2+1\}$, $\textbf b_k+\textbf h_{\alpha_t}\in E$. 
By \eqref{5.2.Disjoint}, for any $k\in\N$ and any $j\in\{1,...,d\}$, there exist at most $k$ natural numbers $t$ for which $b_{k,j}+h_{\alpha_t}^{(j)}\in A_k$. Similarly, for any distinct $i,j\in\{1,...,d\}$, one has $(b_{k,j}-b_{k,i})+(h_{\alpha_t}^{(j)}-h_{\alpha_t}^{(i)})\in A_k$ for at most $k$ natural numbers $t$.\\
We claim that there exists $t\in\{1,...,kd^2+1\}$ such that for any $j\in\{1,...,d\}$, $b_{k,j}+h_{\alpha_t}^{(j)}\not\in A_k$ and  for any $i\neq j$, $(b_{k,j}-b_{k,i})+(h_{\alpha_t}^{(j)}-h_{\alpha_t}^{(i)})\not\in A_k$. Suppose for contradiction that this is not the case. Since there are $d^2-d$ pairs $(i,j)$ with distinct $i, j\in\{1,...,d\}$, there exist at least $k+1$ natural numbers $t$ for which, say,  $b_{k,1}+h_{\alpha_t}^{(1)}\in A_k$, a contradiction.\\
Thus, there exists a sequence $(k_t)_{t\in\N}$ in $\N$ for which the sequences $$(b_{t,j}+h_{\alpha_{k_t}}^{(j)})_{t\in\N},\,j\in\{1,...,d\}$$
are non-degenerated and essentially distinct, 
 and 
$$\{(b_{t,1}+h_{\alpha_{k_t}}^{(1)},...,b_{t,d}+h_{\alpha_{k_t}}^{(d)})\,|\,t\in\N\}\subseteq E.$$\\
Now let $m>1$. By \cref{5.1.SigmaInEveryIP} there exist non-degenerated and essentially distinct sequences $(\textbf f_k^{(j)})_{k\in\N}=(f^{(j)}_{k,1},...,f^{(j)}_{k,m-1})_{k\in\N}$, $j\in\{1,...,d\}$, with the property that $\{(f^{(1)}_{\alpha},...,f^{(d)}_\alpha)\,|\,\alpha\in\N^{(m-1)}\}\subseteq D$. For each $k\in\N$, let 
\begin{equation}\label{5.2.E_kDefn}
E_k=\bigcap_{\alpha\subseteq\{1,...,k+m-1\},\,|\alpha|=m-1}(E-(f^{(1)}_\alpha,...,f^{(d)}_\alpha)).
\end{equation}
By \eqref{5.2.IPReccurence}, for each $k\in\N$, $\overline d_{(F_k)}(E_k)>0$. It follows from the case $m=1$, that  there exist sequences 
$$(g_{k,j})_{k\in\N},\,j\in\{1,...,d\}$$
with the properties that (a) for any $k\in\N$, $(g_{k,1},...,g_{k,d})\in E_k$, (b) for any $j\in\{1,...,d\}$, $\lim_{k\rightarrow\infty}g_{k,j}=\infty$  and  (c) for any distinct $i,j\in\{1,...,d\}$, $\lim_{k\rightarrow\infty}g_{k,i}-g_{k,j}=\infty$. For each $j\in\{1,...,d\}$ form the sequence
$$(\textbf g_k^{(j)})_{k\in\N}=(f_{k,1}^{(j)},...,f_{k,m-1}^{(j)},g_{k,j})=(g_{k,1}^{(j)},...,g_{k,m}^{(j)}).$$
By \eqref{5.2.E_kDefn} and (a), we have that for any $k\in\N$ and any $\alpha\subseteq \{1,...,k-1\}$ with $|\alpha|=m-1$, $(g_{k,1},...,g_{k,d})+(f^{(1)}_\alpha,...,f^{(d)}_\alpha)\in E$ and hence 
$$\{(f_{\{k_1,...,k_{m-1}\}}^{(1)}+g_{k_m,1},...,f_{\{k_1,...,k_{m-1}\}}^{(d)}+g_{k_m,d})\,|\,k_1<\cdots<k_{m-1}<k_m\}\subseteq E.$$
By (b) and 
(c), the sequences $(\textbf g_k^{(1)})_{k\in\N}$,...,$(\textbf g_k^{(d)})_{k\in\N}$ are non-degenerated and essentially distinct. We are done.
\end{proof}
\begin{cor}\label{5.2.Sigma_ell^*HasDensityOne}
Let $(G,+)$ be a countable abelian group and let $d,m\in\N$. Every $\tilde \Sigma_m^*$ set in $G^d$ has uniform density one. 
\end{cor}
\begin{proof}
We will assume that  $D\subseteq G^d$ does not have uniform density one and show that  $D$ is not a $\tilde\Sigma_m^*$ set. Indeed, if $D$ does not have uniform density one, then there exists a F{\o}lner sequence $(F_k)_{k\in\N}$ in $G^d$ for which $\overline d_{(F_k)}(D)<1$. Let $E=G^d\setminus D$ and note that $\overline d_{(F_k)}(E)>0$. By \cref{5.2.FiniteSumsInPositiveDensitySets},  $E$ contains a $\tilde\Sigma_m$ set. This implies that $D$ is not a $\tilde\Sigma_m^*$.
\end{proof}

\subsection{The ubiquity of $\tilde\Sigma_m^*$ sets}

In this section we will show that there exists a broad class of subgroups of $G^d$ with  the property that for each group $H$ from this class, any $\tilde\Sigma_m^*$ set in $G^d$ has a large intersection with $H$. In fact, we will show that either a subgroup $H$ belongs to this class or $G^d\setminus H$ is a $\tilde\Sigma_m^*$ set for any $m\in \N$. 
\begin{defn}
Let $(G,+)$ be a countable abelian group, let $d\in\N$ and let $H$ be a subgroup of $G^d$.  We say that $H$ is an \textbf{admissible subroup of $G^d$} if there exist non-degenerated and essentially distinct sequences $(g_k^{(1)})_{k\in\N}$,...,$(g_k^{(d)})_{k\in\N}$ in $G$ such that 
$$\{(g_k^{(1)},...,g_k^{(d)})\,|\,k\in\N\}\subseteq H.$$
\end{defn}
\begin{example}
Let $(G,+)$ be a countable abelian group and let $H=\{(g,h,0)\,|\,g,h\in G\}\subseteq G^3$. Clearly, $H$ is not an admissible subgroup of $G^3$.
\end{example}
\begin{example}
Let $(G,+)$ be a countable abelian group with an element $g$ of infinite order. For any $d\in\N$ and any distinct $a_1,...,a_d\in\Z\setminus\{0\}$, the set $\{(ka_1 g,ka_2 g,...,ka_dg)\,|\,k\in \Z\}$  is an admissible subgroup of  $G^d$.
\end{example}
\begin{example}
Let $(G,+)$ be a countable abelian torsion group (i.e. each of its elements has finite order). There exists a sequence $(g_k)_{k\in \N}$ in $G$ and a nested sequence of finite subgroups $(G_N)_{N\in\N}$ with the properties: (i)  $G_N$ is generated by $\{g_1,...,g_N\}$ and (ii) for each $k\in\N$, $g_{k+1}\not\in G_k$.
Then for any $d\in\N$ and any distinct $a_1,...,a_d\in\N$, the  group generated by the set $\{(g_{a_1k},g_{a_2k},...,g_{a_dk})\,|\,k\in\N\}$ is  an admissible subgroup of $G^d$. Indeed, note that for any $k\in\N$ and any $a,b\in\N$ with $a<b$, $g_{ak}\not\in G_{ak-1}$ and $(g_{bk}-g_{ak})\not\in G_{ak}$. So $\lim_{k\rightarrow\infty}g_{ak}=\infty$ and $\lim_{k\rightarrow\infty}(g_{bk}-g_{ak})=\infty$.
\end{example}
The following proposition provides a useful characterization of admissible subgroups. 
\begin{prop}\label{5.3.CharacterizationOfAdmissible}
Let $(G,+)$ be a countable abelian group, let $d\in\N$ and let $H$ be a subgroup of $G^d$. The following statements are equivalent:
\begin{enumerate}[(i)]
    \item $H$ is an admissible subgroup of $G^d$.
      \item  There exist an $m\in\N$ and a  $\tilde\Sigma_m$ set $E\subseteq G^d$ such that $E\subseteq H$.
    \item For  any $m\in\N$, there exists a  $\tilde\Sigma_m$ set $E\subseteq G^d$ such that $E\subseteq H$.
    \item There exists an \text{\rm{$\tilde{\text{IP}}$}} set $E\subseteq G^d$ such that $E\subseteq H$.
    \item For any $j\in\{1,...,d\}$, $\pi_j(H)$ is infinite and for any $i\neq j$, $(\pi_j-\pi_i)(H)$ is also infinite, where for each $j\in\{1,...,d\}$, $\pi_j:H\rightarrow G$ is defined by $\pi_j(g_1,...,g_d)=g_j$.
\end{enumerate}
\end{prop}
\begin{proof}
It is not hard to see that (i) and (ii) are equivalent. The implications (i)$\implies$(iii), (iii)$\implies$(iv) and (iv)$\implies$(v) are trivial.  We will now prove  (v)$\implies$(i).\\
Let $P=\{\pi_j\,|\,j\in\{1,...,d\}\}\cup\{\pi_j-\pi_i\,|\,i,j\in\{1,...,d\},\,i\neq j\}$ and let $M$ be the largest non-negative integer for which there exist an $F\subseteq P$ with $|F|=M$ and a sequence $(\textbf g_k)_{k\in\N}$ in $H$ such that for any $\pi\in F$, 
$\lim_{k\rightarrow\infty}\pi(\textbf g_k)=\infty$.  Since $|P|=d^2$, we have $M\leq d^2$. Also, since for each $\pi\in P$, $\pi(H)$ is infinite, $M\geq 1$. If $M=d^2$, then (i) holds. So, assume for contradiction that $M<d^2$.\\
By the definition of $M$, there exists a set $F_0\subseteq P$ with $|F_0|=M$ and a sequence $(\textbf g_k)_{k\in\N}$ in $H$ such that if $\pi\in F_0$, $\lim_{k\rightarrow\infty}\pi(\textbf g_k)=\infty$ and if $\pi\in (P\setminus F_0)$, then there exists a finite set $A_\pi\subseteq G$ such that $\{\pi(\textbf g_k)\,|\,k\in\N\}\subseteq A_\pi$. By passing, if needed, to a subsequence, we can assume that for each $\pi\in (P\setminus F_0)$, there exists a $g_\pi\in G$ such that $\lim_{k\rightarrow\infty}\pi(\textbf g_k)=g_\pi$. Let $\pi_0\in (P\setminus F_0)$. By (v), there exists a sequence $(\textbf g'_k)_{k\in\N}$ in $H$ such that $\lim_{k\rightarrow\infty}\pi_0(\textbf g'_k)=\infty$. Note that for any finite set $A\subseteq H$, any $\pi\in F_0$ and any $t\in\N$, there exists a $k\in\N$ such that for any $k'>k$,
$$\pi(\textbf g_{k'}+\textbf g'_t)=\pi(\textbf g_{k'})+\pi(\textbf g'_t)\not\in A.$$
Also, note that  there exists a $k_0\in\N$ such that for any $k>k_0$, $\pi_0(\textbf g_k)=g_{\pi_0}$.
It follows that we can find an increasing sequence $(k_t)_{t\in\N}$ in $\N$ for which $\lim_{t\rightarrow\infty}\pi(\textbf g_{k_t}+\textbf g'_t)=\infty$ for each $\pi\in F_0\cup\{\pi_0\}$. This contradicts the definition of $M$, completing the proof. 
\end{proof}
\begin{cor}
Let $(G,+)$ be a countable abelian group and let $d\in\N$. A subgroup $H$ of $G^d$ is either   admissible or for any $m\in\N$, $G^d\setminus H$ is a $\tilde\Sigma_m^*$ set.
\end{cor}
\begin{proof}
If $H$ is not an admissible subgroup, \cref{5.3.CharacterizationOfAdmissible}, (ii), implies that for each $m\in\N$, $H$ does not contain any  $\tilde\Sigma_m$ set in $G^d$. Thus, $G^d\setminus H$ is a $\tilde\Sigma_m^*$ set for each $m\in\N$. 
\end{proof}
Before stating and proving one of the main results of this subsection  which deals with the ubiquity of $\tilde\Sigma_m^*$ sets in admissible subgroups (\cref{5.3.SigmaOnAdmisibleSubgroups} below), we need one more definition and a technical lemma.
\begin{defn}\label{5.3.DefnSigma_m*InH}
Let $(G,+)$ be a countable abelian group, let $d,m\in\N$ and let $H\subseteq G^d$ be an admissible subgroup. A set $E\subseteq H$ is called an \rm{$H$-}$\tilde\Sigma_m^*$ set if it has a non-trivial intersection with every $\tilde\Sigma_m$ set contained in  $H$. Similarly, a set $E\subseteq H$ is called an \rm{$H$-}\text{\rm{$\tilde{\text{IP}}\rm{^*}$}} set if it has a non-trivial intersection with every \text{\rm{$\tilde{\text{IP}}$}} set contained in $H$. 
\end{defn}
\begin{rem} 
Let $(G,+)$ be a countable abelian group, let $d\in\N$ and let $H\subseteq G^d$ be an admissible subgroup of $G^d$. It is useful to percieve $H$-$\tilde\Sigma_m^*$ sets as relative versions of $\tilde\Sigma_m^*$ sets in $G^d$. Note that  if $H$ is a proper subgroup of $G^d$, $H$-$\tilde\Sigma_m^*$ sets are not $\tilde\Sigma_m^*$. Indeed,  since for each  $m\in\N$, any translation of a $\tilde\Sigma_m$ set in $G^d$ is again a $\tilde\Sigma_m$ set,  every coset of $H$ contains a $\tilde\Sigma_m$ set in $G^d$. It follows that  $G^d\setminus H$ contains a $\tilde\Sigma_m$ set for each $m\in\N$. Hence, no \rm{$H$-}$\tilde\Sigma_m^*$ set is a  $\tilde\Sigma_m^*$ set.
\end{rem}
\begin{rem}\label{5.3.SigmaOnLateralClasses}
Let $(G,+)$ be a countable abelian group, let $d,m\in\N$, let $H\subseteq G^d$ be an admissible subgroup and let $E$ be a $\tilde\Sigma_m^*$ set in $G^d$. It follows from the  definition that $E\cap H$ is a \rm{$H$-}$\tilde\Sigma_m^*$ set. Indeed, let $D\subseteq H$ be a $\tilde\Sigma_m$ set. We have $(E\cap H)\cap D=E\cap  D\neq \emptyset$. Note also that for any $\textbf g\in G^d$, $E\cap (\textbf g+H)$ is the translation of the \rm{$H$-}$\tilde\Sigma_m^*$ set $(-\textbf g+E)\cap H$. Thus, the cosets of $H$ have a large intersection with $E$ as well.
\end{rem}
\begin{lem}\label{5.3.FiniteSumsInPositiveDensitySetsInH}
Let $(G,+)$ be a countable abelian group, let $d,m\in\N$, let $H$ be an admissible subgroup of $G^d$ and let $(F_k)_{k\in\N}$ be a F{\o}lner sequence in $H$. Any $E\subseteq H$ with $\overline d_{(F_k)}(E)>0$ contains a $\tilde\Sigma_m$ set.
\end{lem}
\begin{proof}
Since $H$ is admissible, there exists an   \text{\rm{$\tilde{\text{IP}}$}} set $D'\subseteq H$. The result in question follows by replacing $D$ by $D'$ in the proof of \cref{5.2.FiniteSumsInPositiveDensitySets} and applying an adequate modification of \cref{5.2.IPPoincare}. 
\end{proof}
\begin{thm}\label{5.3.SigmaOnAdmisibleSubgroups}
Let $(G,+)$ be a countable abelian group, let $d,m\in\N$ and let $H\subseteq G^d$ be an admissible subgroup. Any \rm{$H$-}$\tilde\Sigma_m^*$ set is an \rm{$H$-}\text{\rm{$\tilde{\text{IP}}\rm{^*}$}} set and has uniform density one in $H$.  
\end{thm}
\begin{proof}
Let $E'\subseteq H$ be an \rm{$H$-}$\tilde\Sigma_m^*$ set. By \cref{5.1.SigmaInEveryIP}, every \text{\rm{$\tilde{\text{IP}}$}} set contains a $\tilde\Sigma_m$ set. It follows that $E'$ is an \rm{$H$-}\text{\rm{$\tilde{\text{IP}}\rm{^*}$}} set. By \cref{5.3.FiniteSumsInPositiveDensitySetsInH}, we can argue as in the proof of  \cref{5.2.Sigma_ell^*HasDensityOne} to show that $E'$ has uniform density one in $H$. 
\end{proof}
\begin{cor}
Let $(G,+)$ be a countable abelian group, let $d\in \N$, let $H$ be an admissible subgroup of $G^d$  and let $(X,\mathcal A,\mu,(T_g)_{g\in G})$ be a strongly mixing system.  For any $\textbf g\in G^d$, each set of the form $R_\epsilon(A_0,...,A_\ell)\cap(\textbf g+ H)$ is the translation of a set with uniform density one in $H$.
\end{cor}
\begin{proof}
This result follows from \cref{3.MainResult}, \Cref{5.3.SigmaOnLateralClasses} and \cref{5.3.SigmaOnAdmisibleSubgroups}.
\end{proof}

A natural class of admissible subgroups in $G^d$ is provided by the one-parameter subgroups of the form $$H_{\phi_1,...,\phi_d}=\{(\phi_1(g),...,\phi_d(g))\,|\,g\in G\},$$ where  $\phi_1,...,\phi_d:G\rightarrow G$ are homomorphisms such that for any $j\in\{1,...,d\}$, $|\ker(\phi_j)|<\infty$ and for any $i\neq j$, $|\ker(\phi_j-\phi_i)|<\infty$. The following proposition, alluded to in \Cref{4.CombinatorialRemark}, involves preimages of sets in $G^d$ via the elements of $H_{\phi_1,...,\phi_d}$  and provides an alternative proof of \cref{4.InjectiveDiagonalResult}.  
\begin{prop}\label{5.3.UsefulInSection4}
Let $(G,+)$ be a countable abelian group, let $d,m\in\N$ and let $\phi_1,...,\phi_d:G\rightarrow G$ be homomorphisms such that for any $j\in\{1,...,d\}$, $\ker(\phi_j)$ is finite and for any $i\neq j$, $\ker(\phi_j-\phi_i)$ is also finite. If $E\subseteq G^d$ is a $\tilde\Sigma_m^*$ set, then $E'=\{g\in G\,|\,(\phi_1(g),...,\phi_d(g))\in E\}$ is a $\Sigma_m^*$ set in $G$.
\end{prop}
\begin{proof}
Let $D\subseteq G$ be the $\Sigma_m$ set in $G$ generated by the  non-degenerated sequence $(\textbf g_k)_{k\in\N}=(g_{k,1},...,g_{k,m})_{k\in\N}$ in $G^m$ (i.e. $D=\{g_\alpha\,|\,\alpha\in\N^{(m)}\}$). We will show that $D\cap E'\neq\emptyset$.\\   
By our assumption on $\phi_1,...,\phi_d$, for each $j\in\{1,...,m\}$, the sequences $(\phi_1(g_{k,j}))_{k\in\N}$,....,$(\phi_d(g_{k,j}))_{k\in\N}$ are non-degenerated and essentially distinct. Thus, the set $D'=\{(\phi_1(g_\alpha),...,\phi_d(g_\alpha))\,|\,\alpha\in\N^{(m)}\}$ is a $\tilde\Sigma_m$ set in $G^d$. Noting that $D'\cap E\neq\emptyset$, we obtain $D\cap E'\neq\emptyset$.
\end{proof}

So far we have been focusing on the massivity and ubiquity of \textit{general} $\tilde\Sigma_\ell^*$ sets. However the "dynamical" $\tilde\Sigma_\ell^*$ sets  $R_\epsilon(A_0,...,A_\ell)$, are even more prevalent in $G^\ell$. For example, assuming for convenience that $G=\Z$, one can show that  the sets of the form $R_\epsilon(A_0,...,A_\ell)$ have an ample presence in "polynomial" subsets of $\Z^\ell$. This is illustrated by the following polynomial extension  of \cref{4.ZExample} (which is proved in a companion paper \cite{BerZel-StronglyMixingPET}).
\begin{thm}\label{5.3.StronglyMixingPet}
Let $\ell\in\N$ and let $p_1,...,p_\ell\in \Z[x]$ be non-constant polynomials such that for any distinct $i,j\in\{1,...,\ell\}$, $\deg(p_j-p_i)>0$. There exists an $m\in\N$ such that for any strongly mixing system  $(X,\mathcal A,\mu, T)$, 
any $\epsilon>0$ and any $A_0,...,A_\ell\in\mathcal A$, the set 
\begin{equation}\label{5.3.SetOfPolyReturns}
R_\epsilon^{p_1,...,p_\ell}(A_0,...,A_\ell)=\{n\in\Z\,|\,|\mu(A_0\cap T^{p_1(n)}A_1\cap\cdots\cap T^{p_\ell(n)}A_\ell)-\prod_{j=0}^\ell\mu(A_j)|<\epsilon\}
\end{equation}
is $\Sigma_m^*$.
\end{thm}
The following proposition shows that, in general, $\tilde\Sigma_\ell^*$ sets, unlike the sets of the form $R_\epsilon(A_0,...,A_\ell)$, can be disjoint from the polynomial sets $H_{p_1,...,p_\ell}=\{(p_1(n),...,p_\ell(n))\,|\,n\in\Z\}$, where $p_1,...,p_\ell\in\Z[x]$.
\begin{prop}\label{5.3.PolynomialPathsAreSmall}
Let $\ell\in\N$ and let $p_1,...,p_\ell\in \Z[x]$ be non-constant polynomials such that  for any distinct  $i,j\in\{1,...,\ell\}$, $\deg(p_j-p_i)>0$. Suppose that $\deg(p_1)>1$. Then, for any $m\geq 2$, $H_{p_1,...,p_\ell}$ contains no $\tilde\Sigma_m$ sets. Equivalently, $\Z^\ell\setminus H_{p_1,...,p_\ell}$ is a $\tilde\Sigma_m^*$ set for each $m\geq 2$.
\end{prop}
\begin{proof}
Since the projection onto the first coordinate of any $\tilde\Sigma_m$ set $E\subseteq \Z^\ell$ is a $\Sigma_m$ set in $\Z$, it suffices to show that the set $\{p_1(n)\,|\,n\in\Z\}$ contains no $\Sigma_m$ sets. Suppose for contradiction that $\{p_1(n)\,|\,n\in\Z\}$ contains a $\Sigma_m$ set 
$$D=\{n_{k_1}^{(1)}+\cdots+n_{k_m}^{(m)}\,|\,k_1<\cdots<k_m\},$$
where $(n_k^{(1)})_{k\in\N}$,...,$(n_k^{(m)})_{k\in\N}$ are non-degenerated sequences in $\Z$.\\ 
Choose $t_1,t_2,t_3\in\N$ to be such that $n^{(1)}_{t_1}<n^{(1)}_{t_2}<n^{(1)}_{t_3}$ and let $$I=\{n_{t_1}^{(1)}+n_{k_2}^{(2)}+\cdots+n_{k_m}^{(m)}\,|\,\max\{t_1,t_2,t_3\}<k_2<\cdots<k_m\}.$$
Clearly $I$ is an infinite subset of $D$. So, letting $a=n_{t_2}^{(1)}-n_{t_1}^{(1)}$ and  $b=n_{t_3}^{(1)}-n_{t_1}^{(1)}$, we have $a+I\subseteq D$ and $b+I\subseteq D$.\\
Let $(n_k)_{k\in\N}$ be an enumeration of the elements of $I$. One can find an increasing sequence $(k_j)_{j\in\N}$ for which at least two of the sets $\{n_{k_j}\,|\,j\in\N\}$, $\{a+n_{k_j}\,|\,j\in\N\}$ and $\{b+n_{k_j}\,|\,j\in\N\}$ are  contained in at least one of the sets $\{p_1(n)\,|\,n\in\N\}$ and $\{p_1(-n)\,|\,n\in\N\}$. We will assume that $\{a+n_{k_j}\,|\,j\in\N\}$ and $\{b+n_{k_j}\,|\,j\in\N\}$ are contained in $\{p_1(n)\,|\,n\in \N\}$ (the other cases can be handled similarly). 
It follows that there exist infinitely many  pairs $(n,m)\in\N\times \N$ such that $p_1(n)-p_1(m)=b-a$. Since $b>a$, this contradicts the fact that $\deg(p_1)>1$.
\end{proof}


\section{Multiple recurrence for mildly and weakly mixing systems via $\mathcal R$-limits}
As we saw above, $\mathcal R$-limits are adequate for characterizing strong mixing and obtaining higher order mixing properties. In this section, we will show that $\mathcal R$-limits can be also useful in dealing with mildly and weakly mixing systems. In particular, we will obtain analogues of \cref{3.MainResult} for midly and weakly mixing systems.

\subsection{Mildly mixing systems}
    In this subsection we will deal with mildly mixing systems (introduced in \Cref{6.1.MildlyMixingDefn} below) from the perspective of $\mathcal R$-limits. The notion of mild mixing has multiple equivalent forms (see \cite{walters1972someMildMixing}, \cite{walters1982MildMixingEquivalence} and \cite{FurWeissMildMixing}) and plays a fundamental role in IP ergodic theory, including various refinements of the classical Szemer{\'e}di theorem (see \cite{berMcCuIPPolySzemeredi} and \cite{FKIPSzemerediLong}).  The multiple recurrence theorems for mildly mixing systems (see \cite{FBook} and \cite{FKIPSzemerediLong}) utilize the notion of IP-limit which we will presently define.  We will then establish a connection between   IP-limits and   $\mathcal R$-limits and, finally, prove an analogue of \cref{3.MainResult} for mildly mixing actions.
\begin{defn}
(Cf. \cite[Definitions 1.1 and 1.3]{FKIPSzemerediLong}) Let $(X,d)$ be a compact metric  space and let $(x_\alpha)_{\alpha\in\mathcal F}$ be an $\mathcal F$-sequence in $X$. A set $\mathcal F^{(1)}\subseteq \mathcal F$ is an \text{\rm{IP}}-ring if there exists a sequence $(\alpha_k)_{k\in\N}$ in $\mathcal F$ with $\alpha_k<\alpha_{k+1}$ for each $k\in\N$ for which 
$$\mathcal F^{(1)}=\{\bigcup_{j\in\alpha}\alpha_j\,|\,\alpha\in\mathcal F\}.$$
For any  \text{\rm{IP}}-ring $\mathcal F^{(1)}$, we write 
$$\mathop{\text{\rm{IP-lim}}}_{\alpha\in\mathcal F^{(1)}}x_\alpha=x$$
if for every $\epsilon>0$, there exists an $\alpha_0\in\mathcal F^{(1)}$ such that for any $\alpha\in\mathcal F^{(1)}$ with $\alpha>\alpha_0$, 
$$d(x_\alpha,x)<\epsilon.$$
\end{defn}
It follows from a result of Hindman \cite{HIPPartitionRegular}  that if $(x_\alpha)_{\alpha\in\mathcal F}$ is an $\mathcal F$-sequence in a compact metric space $X$, then for any \text{\rm{IP}}-ring $\mathcal F^{(1)}\subseteq \mathcal F$ one can always find an $x\in  X$ and an \text{\rm{IP}}-ring $\mathcal F^{(2)}\subseteq \mathcal F^{(1)}$ such that 
\begin{equation}\label{6.1.IPAlwaysExists}
\mathop{\text{\rm{IP-lim}}}_{\alpha\in\mathcal F^{(2)}}x_\alpha=x
\end{equation}
(see \cite[Theorem 8.14]{FBook}).
In particular,  for any countable abelian group $(G,+)$, any  sequence $(g_k)_{k\in\N}$ in $G$ and any probability measure preserving system $(X,\mathcal A,\mu, (T_g)_{g\in G})$, there exists an \text{\rm{IP}}-ring $\mathcal F^{(1)}$ for which 
$$\mathop{\text{\rm{IP-lim}}}_{\alpha\in\mathcal F^{(1)}}T_{g_\alpha}$$
exists in the weak operator topology of $L^2(\mu)$. This implies (and is equivalent to) the fact that for any $A_0,A_1\in\mathcal A$,
$$\mathop{\text{\rm{IP-lim}}}_{\alpha\in\mathcal F^{(1)}}\mu(A_0\cap T_{g_\alpha} A_1)$$
exists.
\begin{thm}\label{6.1.EquivalenceOfConvergenceIP}
Let $(X,d)$ be a compact metric space, let $(G,+)$ be a countable abelian group, let $(x_g)_{g\in G}$ be a sequence in $X$, let $x_0\in X$  and let $(g_k)_{k\in\N}$ be a sequence in $G$. The following statements are equivalent:
\begin{enumerate}[(i)]
    \item For any \text{\rm{IP}}-ring $\mathcal F^{(1)}\subseteq \mathcal F$ for which $\mathop{\text{\rm{IP-lim}}}_{\alpha\in\mathcal F^{(1)}}x_{g_\alpha}$ exists, one has
    \begin{equation}\label{6.1.IPLimitForSubRing}
    \mathop{\text{\rm{IP-lim}}}_{\alpha\in\mathcal F^{(1)}}x_{g_\alpha}=x_0.
    \end{equation}
    \item For any \text{\rm{IP}}-ring $\mathcal F^{(1)}\subseteq \mathcal F$ there exist  an $m\in\N$ and a sequence $(h_{k,1},...,h_{k,m})_{k\in\N}$ in $G^m$ such that $\{h_\alpha\,|\,\alpha\in\N^{(m)}\}\subseteq \{g_\alpha\,|\,\alpha\in\mathcal F^{(1)}\}$ and 
    \begin{equation}\label{6.1.RLimitForSubSequence}
    \rlim{\alpha\in\N^{(m)}}x_{h_\alpha}=x_0.
    \end{equation}
\end{enumerate}
\end{thm}
\begin{proof}
(i)$\implies$(ii): Let $\mathcal F^{(1)}$ be an \text{\rm{IP}}-ring. Since $X$ is compact, we can assume (by  passing, if needed, to a sub \text{\rm{IP}}-ring) that $\mathop{\text{\rm{IP-lim}}}_{\alpha\in\mathcal F^{(1)}}x_{g_\alpha}$ exists. Thus, by (i), \eqref{6.1.IPLimitForSubRing} holds. It follows from the definition of an IP-limit that there exists a sequence $(h_k)_{k\in\N}$ in $G$ such that $\{h_k\,|\,k\in\N\}\subseteq \{g_\alpha\,|\,\alpha\in\mathcal F^{(1)}\}$ and $\lim_{k\rightarrow\infty}x_{h_k}=x_0$. This completes the proof of (i)$\implies$(ii).\\
(ii)$\implies$(i): Let $\mathcal F^{(1)}$ be an \text{\rm{IP}}-ring for which $\mathop{\text{\rm{IP-lim}}}_{\alpha\in\mathcal F^{(1)}}x_{g_\alpha}=y$ for some $y\in X$. Suppose for contradiction that there exists an $\epsilon>0$ for which $d(y,x_0)>\epsilon$.  By the definition of an IP-limit, there exists $\alpha_0\in\mathcal F$ such that for any $\alpha\in\mathcal F^{(1)}$  with $\alpha>\alpha_0$, $d(x_{g_\alpha},x_0)>\epsilon$. Since $\{\alpha\in\mathcal F^{(1)}\,|\,\alpha>\alpha_0\}$ is an \text{\rm{IP}}-ring,  it follows from (ii) that there exist an $m\in\N$ and a sequence $(h_{k,1},...,h_{k,m})_{k\in\N}$ in $G^m$ such that $\{h_\alpha\,|\,\alpha\in\N^{(m)}\}\subseteq \{g_\alpha\,|\,\alpha\in\mathcal F^{(1)}\text{ and }\alpha>\alpha_0\}$ and    $\rlim{\alpha\in\N^{(m)}}x_{h_\alpha}=x_0$. In particular, there exists an $h\in  \{g_\alpha\,|\,\alpha\in\mathcal F^{(1)}\text{ and }\alpha>\alpha_0\}$ for which $d(x_h,x_0)<\epsilon$, a contradiction.
\end{proof}
\begin{rem}
\cref{6.1.EquivalenceOfConvergenceIP} shows that IP-limits can be attained via $\mathcal R$-limits. The following example demonstrates that this is not the case the other way around.  Let $G=\Z$, let $X=\{0,1\}$, let  $m\in\N$, and  consider the $\Sigma_m$ set   $E=\{3^{k_1}+\cdots+3^{k_m}\,|\,k_1<\cdots<k_m\}$. The set $E$ is comprised of all the elements of  $3\N$ whose base 3 expansion has exactly $m$ non-zero entries, all of which are 1. It follows that  there are no $a,b,c\in E$ for which $a+b=c$. This, in turn, implies that $E$ contains no IP sets and hence   $\Z\setminus E$ is an IP$^*$ set. Let $(n_k)_{k\in\N}$ be a sequence in $\Z$ and let $\mathcal F^{(1)}\subseteq\mathcal F$ be an \text{\rm{IP}}-ring for which  $\mathop{\text{\rm{IP-lim}}}_{\alpha\in\mathcal  F^{(1)}}\mathbbm 1_E(n_\alpha)$ exists. Since $0\not\in E$ and $\Z\setminus E$ is IP$^*$, one has $\mathop{\text{\rm{IP-lim}}}_{\alpha\in\mathcal  F^{(1)}}\mathbbm 1_E(n_\alpha)=0$. On the other hand, since for any $k_1<\cdots<k_m$, $\mathbbm 1_E(3^{k_1}+\cdots+3^{k_m})=1$, one has that for any infinite set $S\subseteq \N$,  
$$\rlim{\{k_1,...,k_m\}\in S^{(m)}}\mathbbm 1_E(3^{k_1}+\cdots+3^{k_m})=1.$$
\end{rem}
\begin{defn}\label{6.1.MildlyMixingDefn}
Let $(G,+)$ be a countable abelian group and let $(X,\mathcal A,\mu, (T_g)_{g\in G})$ be a  measure preserving system. $(T_g)_{g\in G}$ is mildly mixing if for any sequence $(g_k)_{k\in\N}$ in $G$ for which $\lim_{\alpha\rightarrow\infty}g_\alpha=\infty$, there exists an \text{\rm{IP}}-ring $\mathcal F^{(1)}$ such that for any $f\in L^2(\mu)$,
\begin{equation}\label{6.1.IPLimitMixing}
\mathop{\text{\rm{IP-lim}}}_{\alpha\in\mathcal F^{(1)}}T_{g_\alpha}f=\int_Xf\text{d}\mu
\end{equation}
weakly.
\end{defn}
We are now ready to state and prove the main theorem of this subsection. It can be viewed as an analogue of \cref{3.MainResult} for mildly mixing actions. We remind the reader that a sequence of measure preserving transformations $(T_k)_{k\in\N}$ of a probability space $(X,\mathcal A,\mu)$ has the mixing property if for every $A_0,A_1\in\mathcal A$, $\lim_{k\rightarrow\infty}\mu(A_0\cap T_k^{-1}A_1)=\mu(A_0)\mu(A_1)$. 
\begin{thm}\label{6.1.GlobalMildlyMixing}
Let $\ell\in\N$, let $(G,+)$ be a countable abelian group and let $(X,\mathcal A,\mu,(T_g)_{g\in G})$ be a measure preserving system. The following statements are equivalent:
\begin{enumerate}[(i)]
    \item $(T_g)_{g\in G}$ is mildly mixing.
    \item For any \text{\rm{$\tilde{\text{IP}}$}} set $E\subseteq G^\ell$  and any $m\in\N$, there exist non-degenerated and essentially distinct sequences $(\textbf g_k^{(j)})_{k\in\N}=(g_{k,1}^{(j)},...,g_{k,m}^{(j)})_{k\in\N}$, $j\in\{1,...,\ell\}$,
    in $G^m$ with the properties:
    \begin{enumerate}[(a)]
    \item $\{(g_\alpha^{(1)},...,g_\alpha^{(\ell)})\,|\,\alpha\in\N^{(m)}\}\subseteq E$. 
    \item For any  $t\in\{1,...,m\}$ and any $j\in\{1,...,\ell\}$, $(T_{g_{k,t}^{(j)}})_{k\in\N}$ has the mixing property.
    \item For any $t$ and any  $i\neq j$, $(T_{g_{k,t}^{(j)}-g_{k,t}^{(i)}})_{k\in\N}$  has the mixing property.
    \end{enumerate}
    \item For any \text{\rm{$\tilde{\text{IP}}$}} set $E\subseteq G^\ell$, there  exist an $m\in\N$ and non-degenerated and essentially distinct sequences $(\textbf g_k^{(1)})_{k\in\N},...,(\textbf g_k^{(\ell)})_{k\in\N}$ in $G^m$ with $\{(g_\alpha^{(1)},...,g_{\alpha}^{(\ell)})\,|\,\alpha\in\N^{(m)}\}\subseteq E$ and such that for any $A_0,...,A_\ell\in\mathcal A$,
    \begin{equation}\label{6.1.RMixingInMain}
    \rlim{\alpha\in \N^{(m)}}\mu(A_0\cap T_{g^{(1)}_\alpha} A_1\cap\cdots\cap T_{g^{(\ell)}_\alpha}A_\ell)=\prod_{j=0}^\ell\mu(A_j).
    \end{equation}
    \item Given sequences $(g^{(1)}_k)_{k\in\N}$,...,$(g^{(\ell)}_k)_{k\in\N}$ in $G$ such 
    that for any $j\in\{1,...,\ell\}$, $\lim_{\alpha\rightarrow\infty}g^{(j)}_\alpha=\infty$ 
    and for any $i\neq j$, 
    $\lim_{\alpha\rightarrow\infty}g_\alpha^{(j)}-g_\alpha^{(i)}=\infty$ (and so 
    $E=\{(g_\alpha^{(1)},...,g_\alpha^{(\ell)})\,|\,\alpha\in\mathcal F\}$
    is an \text{\rm{$\tilde{\text{IP}}$}} set), there exists an \text{\rm{IP}}-ring $\mathcal F^{(1)}$ such that for any $A_0,...,A_\ell \in\mathcal A$,
    \begin{equation}\label{6.1.IPMultiCorrelation}
    \mathop{\text{\rm{IP-lim}}}_{\alpha\in\mathcal F^{(1)}}\mu(A_0\cap T_{g_\alpha^{(1)}}A_1\cap\cdots\cap T_{g_\alpha^{(\ell)}}A_\ell)=\prod_{j=1}^\ell \mu(A_j).
    \end{equation}
    \item For any $A_0,...,A_\ell\in\mathcal A$ and any $\epsilon>0$, the set 
    $$R_\epsilon(A_0,...,A_\ell)=\{(g_1,...,g_\ell)\in G^\ell\,|\,|\mu(A_0\cap T_{g_1}A_1\cap\cdots \cap T_{g_\ell}A_\ell)-\prod_{j=0}^\ell \mu( A_j)|<\epsilon\}$$ 
    is an \text{\rm{$\tilde{\text{IP}}\rm{^*}$}} set. 
\end{enumerate}  
\end{thm}
\begin{proof}
(i)$\implies$(ii): Let $m\in\N$, let $E\subseteq G^\ell$ be an \text{\rm{$\tilde{\text{IP}}$}} set and let the sequences $(h_k^{(1)})_{k\in\N}$,...,$(h_k^{(\ell)})_{k\in\N}$ in $G$ be such that $E=\{(h_\alpha^{(1)},...,h_\alpha^{(\ell)})\,|\,\alpha\in\mathcal F\}$. By the stipulation made in  \Cref{5.1.CommonSenseSequence}, for any \text{\rm{IP}}-ring $\mathcal F^{(1)}\subseteq \mathcal F$, the set $\{(h^{(1)}_\alpha,...,h^{(\ell)}_\alpha)\,|\,\alpha\in\mathcal F^{(1)}\}$ is again an \text{\rm{$\tilde{\text{IP}}$}} set. Pick $\mathcal F^{(1)}$ to be an \text{\rm{IP}}-ring such  that for any $A_0,A_1\in\mathcal A$ and any $i,j\in\{1,...,\ell\}$,
\begin{equation}\label{6.1.IPExpressions1}
\mathop{\text{\rm{IP-lim}}}_{\alpha\in\mathcal F^{(1)}}\mu(A_0\cap T_{h^{(j)}_\alpha}A_1)\text{ and if }i\neq j,\,\mathop{\text{\rm{IP-lim}}}_{\alpha\in\mathcal F^{(1)}}\mu(A_0\cap T_{h^{(j)}_\alpha-h^{(i)}_\alpha}A_1)
\end{equation}
exist. Let $(\alpha_k)_{k\in\N}$ be the sequence in $\mathcal  F$ generating $\mathcal F^{(1)}$ (so, in particular, $\alpha_k<\alpha_{k+1}$ for each $k\in\N$).
It follows from (i) that each of the limits appearing in \eqref{6.1.IPExpressions1} equals $\mu(A_0)\mu(A_1)$ (otherwise, we would have a contradiction with formula \eqref{6.1.IPLimitMixing}). Thus,  for any $A_0,A_1\in\mathcal A$ and any $i,j\in\{1,...,\ell\}$,
$$\lim_{k\rightarrow\infty}\mu(A_0\cap T_{h^{(j)}_{\alpha_k}}A_1)=\mu(A_0)\mu(A_1)\text{ and if }i\neq j,\,\lim_{k\rightarrow\infty}\mu(A_0\cap T_{h^{(j)}_{\alpha_k}-h^{(i)}_{\alpha_k}}A_1)=\mu(A_0)\mu(A_1).$$
For each $j\in\{1,...,\ell\}$, let  $(\textbf g^{(j)}_k)_{k\in\N}=(\underbrace{h_{\alpha_k}^{(j)},...,h_{\alpha_k}^{(j)})}_{m\text{ times}}$. It is now easy to check that the sequences $(\textbf g^{(1)}_k)_{k\in\N}$,...,$(\textbf g^{(\ell)}_k)_{k\in\N}$ are non-degenerated, essentially distinct, and satisfy (a)-(c), completing the proof of (i)$\implies$(ii).\\
(ii)$\implies$(iii): This follows from \cref{2.MainResult}.\\
(iii)$\implies$(iv):  We will prove (iv)  by applying \cref{6.1.EquivalenceOfConvergenceIP} to the $G^\ell$-sequence  
$$x_{(g_1,...,g_\ell)}=\mu(A_0\cap T_{g_1}A_1\cap\cdots\cap T_{g_\ell}A_\ell),\,(g_1,...,g_\ell)\in G^\ell$$
and the sequence $(g_k^{(1)},...,g_k^{(\ell)})_{k\in\N}$ in $G^\ell$.\\ 
Note that for any IP-ring $\mathcal F^{(2)}$, $\{(g_\alpha^{(1)},...,g_\alpha^{(\ell)})\,|\,\alpha\in\mathcal F^{(2)}\}$ is an \text{\rm{$\tilde{\text{IP}}$}} set.
By (iii),  there exist an $m\in\N$ and  non-degenerated and essentially distinct sequences $(\textbf h_k^{(1)})_{k\in\N}$,...,$(\textbf h_k^{(\ell)})_{k\in\N}$ in $G^m$ with 
$$\{(h_\alpha^{(1)},...,h_\alpha^{(\ell)})\,|\,\alpha\in\mathcal \N^{(m)}\}\subseteq \{(g_\alpha^{(1)},...,g_\alpha^{(\ell)})\,|\,\alpha\in\mathcal F^{(2)}\}$$
for which \eqref{6.1.RMixingInMain} holds. Letting $\mathcal F^{(1)}$ be an \text{\rm{IP}}-ring  for which the left-hand side of \eqref{6.1.IPMultiCorrelation} exists for any $A_0,...,A_\ell\in\mathcal A$, we obtain by \cref{6.1.EquivalenceOfConvergenceIP} that \eqref{6.1.IPMultiCorrelation} holds.\\
(iv)$\implies$(v): This implication follows from the definition of \text{\rm{$\tilde{\text{IP}}\rm{^*}$}}.\\
(v)$\implies$(i): Let $(g_k)_{k\in\N}$ be a sequence in $G$ with the property that $\lim_{\alpha\rightarrow\infty}g_\alpha=\infty$. It suffices to show that for some \text{\rm{IP}}-ring $\mathcal F^{(1)}$ and any $A_0,A_1\in\mathcal A$, 
$$\mathop{\text{\rm{IP-lim}}}_{\alpha\in\mathcal F^{(1)}}\mu(A_0\cap T_{g_\alpha}A_1)=\mu(A_0)\mu(A_1).$$
By \eqref{6.1.IPAlwaysExists}, there exists an \text{\rm{IP}}-ring $\mathcal F^{(1)}\subseteq \mathcal F$ such that for any $A_0,A_1\in\mathcal A$,
\begin{equation}\label{6.1.MildMixingExitenceOfIPLimit}
\mathop{\text{\rm{IP-lim}}}_{\alpha\in\mathcal F^{(1)}}\mu(A_0\cap T_{g_\alpha}A_1)
\end{equation}
exists. Let $(\gamma_k)_{k\in\N}$ be a sequence in $\mathcal F^{(1)}$ with $\gamma_k<\gamma_{k+1}$ for each $k\in\N$ and such that the sequences $(h_k^{(j)})_{k\in\N}=(g_{\gamma_{j+\ell k}})_{k\in\N}$, $j\in\{1,...,\ell\}$, in $G$  satisfy (a) for any $j\in\{1,...,\ell\}$, $\lim_{\alpha\rightarrow\infty}h_\alpha^{(j)}=\infty$ and (b) for any $i\neq  j$, $\lim_{\alpha\rightarrow\infty}h_\alpha^{(j)}-h_\alpha^{(i)}=\infty$. For each $\alpha_0\in\mathcal F$, let 
$$E_{\alpha_0}=\{(h^{(1)}_\alpha,...,h^{(\ell)}_{\alpha})\,|\,\alpha\in\mathcal F\text{ and }\alpha>\alpha_0\}.$$
Since $E_{\alpha_0}$ is an \text{\rm{$\tilde{\text{IP}}$}} set, (v) implies that for any $\alpha_0\in\mathcal F$, any $A_0,A_1\in\mathcal A$ and any $\epsilon>0$, 
$$E_{\alpha_0}\cap R_\epsilon(A_0,A_1,X,...,X)\neq\emptyset.$$
Thus, for any $\alpha_0\in\mathcal F$, there exists an $\alpha>\alpha_0$ such that $h^{(1)}_\alpha\in R_\epsilon(A_0,A_1)$. Note that
$$\lim_{\alpha\rightarrow\infty}\min(\bigcup_{k\in\alpha}\gamma_{1+\ell k})=\infty.$$
It follows that for any $\beta_0\in\mathcal F$, there is an   $\alpha\in\mathcal F$ such  that  $h_\alpha^{(1)}\in R_\epsilon(A_0,A_1)$ and such that  $\beta=\bigcup_{k\in\alpha}\gamma_{1+\ell k}\in\mathcal F^{(1)}$ satisfies  $\beta>\beta_0$.  But $g_\beta=g_{(\bigcup_{k\in\alpha}\gamma_{1+\ell k})}=h_\alpha^{(1)}$, so
$$|\mu(A_0\cap T_{g_\beta}A_1)-\mu(A_0)\mu(A_1)|<\epsilon.$$
Since $\epsilon$ was arbitrary,  for any $A_0,A_1\in\mathcal A$, 
$$\mathop{\text{\rm{IP-lim}}}_{\alpha\in\mathcal F^{(1)}}\mu(A_0\cap T_{g_\alpha}A_1)=\mu(A_0)\mu(A_1),$$
which completes the proof.
\end{proof}
\begin{rem}
We saw in Section 4 that the  versatility of $\mathcal R$-limits allows one to obtain from the \textit{multiparameter} \cref{3.MainResult} some interesting results of \textit{diagonal} nature. Similarly, one can obtain diagonal results from \cref{6.1.GlobalMildlyMixing}. For example, let $G=\Z$ and assume that $(X,\mathcal A,\mu,T)$ is a mildly mixing system. Then, by \cref{6.1.GlobalMildlyMixing}, (iv), for any strictly increasing sequence $(n_k)_{k\in\N}$ in $\Z$, any non-zero and distinct integers $a_1,...,a_\ell$ and any \text{\rm{IP}}-ring $\mathcal F^{(1)}\subseteq \mathcal F$, there exists an \text{\rm{IP}}-ring $\mathcal F^{(2)}\subseteq \mathcal F^{(1)}$ such that for any $A_0,...,A_\ell\in\mathcal A$,
\begin{equation}\label{6.1.SzemerediExpresion}
\mathop{\textit{IP-lim}}_{\alpha\in\mathcal F^{(2)}}\mu(A_0\cap T^{a_1n_\alpha}A_1\cap \cdots\cap T^{a_\ell n_\alpha}A_\ell)=\prod_{j=0}^\ell \mu(A_j).
\end{equation}
(Cf. \cite[Theorem 9.27]{FBook} and  \cite[Theorem 5.4]{FKIPSzemerediLong}.)
\end{rem}

\subsection{Weakly mixing systems}
This subsection is devoted to weakly mixing systems (which were introduced in Subsection 5.2) and has a similar structure to that of Subsection 6.1. We will first establish a technical lemma which connects $\mathcal R$-limits with C{\'e}saro convergence. We will then prove an analogue of \cref{3.MainResult} (see \cref{6.2.GlobalWeaklyMixing}  below) for weakly mixing systems and derive a corollary which has  diagonal nature.    
\begin{lem}\label{6.2.LimitEquivalenceWM}
Let $(G,+)$ be a countable abelian group, let $(X,d)$ be a compact metric space, let $(x_g)_{g\in G}$ be a sequence in $X$, let $x_0\in X$, let $(F_k)_{k\in\N}$ be a F{\o}lner sequence in $G$ and let $E\subseteq G$ be such that $\overline d_{(F_k)}(E)>0$. The following statements are equivalent:
\begin{enumerate}[(i)]
    \item 
    \begin{equation}\label{6.2.DensityLimit}
        \lim_{k\rightarrow\infty}\frac{1}{|F_k|}\sum_{g\in F_k}\mathbbm 1_E(g)d(x_g,x_0)=0.
    \end{equation}
    \item For any $D\subseteq  E$ with $\overline d_{(F_k)}(D)>0$, there exist an $m\in\N$ and a sequence $(g_{k,1},...,g_{k,m})_{k\in\N}$ in $G^m$ for which $\{g_\alpha\,|\,\alpha\in\N^{(m)}\}\subseteq D$ and 
    \begin{equation}\label{6.2.RlimitDensityResult}
    \rlim{\alpha\in\N^{(m)}}x_{g_\alpha}=x_0.
    \end{equation}
\end{enumerate}
\end{lem}
\begin{proof}
(i)$\implies$(ii):  Let $D\subseteq E$ be such that $\overline d_{(F_k)}(D)>0$. It follows from  \eqref{6.2.DensityLimit} that 
$$\lim_{k\rightarrow\infty}\frac{1}{|F_k|}\sum_{g\in F_k}\mathbbm 1_D(g)d(x_g,x_0)=0.$$
Let $\epsilon>0$. There exist infinitely many $g\in D$ such that $d(x_g,x_0)<\epsilon$ (otherwise, we would have 
$\limsup_{k\rightarrow\infty}\frac{1}{|F_k|}\sum_{g\in F_k}\mathbbm 1_D(g)d(x_g,x_0)>0$). Thus, for each $k\in\N$, there is a $g_k\in D$ with $d(x_{g_k},x_0)<\frac{1}{k}$. It follows now that
$$\rlim{\{k\}\in\N^{(1)}}x_{g_{\{k\}}}=\lim_{k\rightarrow\infty}x_{g_k}=x_0.$$
(ii)$\implies$(i): It suffices to show that for any given $\epsilon>0$, $\overline d_{(F_k)}(D_\epsilon)=0$, where 
$$D_\epsilon=\{g\in E\,|\,d(x_g,x_0)>\epsilon\}.$$ 
(This will imply that for each $\epsilon>0$, $$\limsup_{k\rightarrow\infty}\frac{1}{|F_k|}\sum_{g\in F_k}\mathbbm 1_E(g)d(x_g,x_0)\leq\limsup_{k\rightarrow\infty}(\frac{1}{|F_k|}\sum_{g\in F_k}\epsilon\mathbbm 1_{E\setminus D_\epsilon}(g)+\frac{1}{|F_k|}\sum_{g\in F_k}\mathbbm 1_{D_\epsilon}(g)d(x_g,x_0))\leq \epsilon.)$$ Fix $\epsilon>0$ and suppose for contradiction that $\overline d_{(F_k)}(D_\epsilon)>0$. It follows from (ii) that there exist an $m\in\N$ and a  sequence $(g_{k,1},...,g_{k,m})_{k\in\N}$ in $G^m$ with $\{g_\alpha\,|\,\alpha\in\N^{(m)}\}\subseteq D_\epsilon$ for which \eqref{6.2.RlimitDensityResult} holds. In particular, for some $g\in D_\epsilon$, $d(x_g,x_0)<\epsilon$, which gives us the desired contradiction.
\end{proof}

We collect in the following proposition some equivalent definitions of weak mixing which will be needed for the proof of \cref{6.2.GlobalWeaklyMixing} below. The proof is totally analogous to the classical case $G=\Z$ and is omitted.
\begin{prop}\label{6.2.EquivalentFormsOFWM}
Let $(G,+)$ be a countable abelian group and let  $(X,\mathcal A,\mu, (T_g)_{g\in G})$ be a measure preserving system. The following statements are equivalent:
\begin{enumerate}[(i)]
    \item $(T_g)_{g\in G}$ is weakly mixing.
    \item For any ergodic probability measure preserving system $(Y,\mathcal B,\nu, (S_g)_{g\in G})$,  the system 
    $$(X\times Y,\mathcal A\otimes \mathcal B,\mu\otimes\nu, (T_g\times S_g)_{g\in G})$$
    is ergodic.
    \item For any F{\o}lner sequence $(F_k)_{k\in\N}$ in $G$ there exists a set $B\subseteq G$ with $\overline d_{(F_k)}(B)=0$ such that for any $A_0,A_1\in\mathcal A$,
    $$\lim_{g\rightarrow\infty,\,g\not\in B}\mu(A_0\cap T_gA_1)=\mu(A_0)\mu(A_1).$$
    \item There exists a sequence $(g_k)_{k\in\N}$ in $G$ with $\lim_{k\rightarrow\infty}g_k=\infty$ such that for any $A_0,A_1\in\mathcal A$,
    $$\lim_{k\rightarrow\infty}\mu(A_0\cap T_{g_k}A_1)=\mu(A_0)\mu(A_1).$$
\end{enumerate}
\end{prop}
\begin{rem}\label{6.2.ProductOfWeaklyMixingIsWeaklyMixing}
It follows from (ii) that for any two weakly mixing systems  $(X,\mathcal A,\mu, (T_g)_{g\in G})$ and $(Y,\mathcal B,\nu, (S_g)_{g\in G})$, $(T_g\times S_g)$ is again weakly mixing. 
\end{rem}

\begin{thm}\label{6.2.GlobalWeaklyMixing}
Let $\ell\in\N$, let $(G,+)$ be a countable abelian group and let $(X,\mathcal A,\mu,(T_g)_{g\in G})$ be a measure preserving system. The following statements are equivalent:
\begin{enumerate}[(i)]
    \item $(T_g)_{g\in G}$ is weakly mixing.
    \item For any F{\o}lner sequence $(F_k)_{k\in\N}$ in $G^\ell$, any set $E\subseteq G^\ell$ with $\overline d_{(F_k)}(E)>0$ and any $m\in\N$, there exist non-degenerated and essentially distinct sequences $(\textbf g_k^{(j)})_{k\in\N}=(g_{k,1}^{(j)},...,g_{k,m}^{(j)})_{k\in\N}$, $j\in\{1,...,\ell\}$,
    in $G^m$
    with the properties:   
    \begin{enumerate}[(a)]
    \item $\{(g_\alpha^{(1)},...,g_\alpha^{(\ell)})\,|\,\alpha\in\N^{(m)}\}\subseteq E$, 
    \item For any  $t\in\{1,...,m\}$ and any $j\in\{1,...,\ell\}$, $(T_{g_{k,t}^{(j)}})_{k\in\N}$ has the mixing property and  
    \item For any $t$ and any  $i\neq j$, $(T_{g_{k,t}^{(j)}-g_{k,t}^{(i)}})_{k\in\N}$  has the mixing property.
    \end{enumerate}
    \item For any F{\o}lner sequence $(F_k)_{k\in\N}$ in $G^\ell$ and any set $E\subseteq G^\ell$ with $\overline  d_{(F_k)}(E)>0$, there exist an $m\in\N$ and sequences $(\textbf g_k^{(1)})_{k\in\N},...,(\textbf g_k^{(\ell)})_{k\in\N}$
    in $G^m$ with $\{(g_\alpha^{(1)},...,g_{\alpha}^{(\ell)})\,|\,\alpha\in\N^{(m)}\}\subseteq E$ and such that for any $A_0,...,A_\ell\in\mathcal A$,
    $$\rlim{\alpha\in \N^{(m)}}\mu(A_0\cap T_{g^{(1)}_\alpha} A_1\cap\cdots\cap T_{g^{(\ell)}_\alpha}A_\ell)=\prod_{j=0}^\ell\mu(A_j).$$
    \item For any $A_0,...,A_\ell\in\mathcal A$ and any $\epsilon>0$, the set
    $$R_\epsilon(A_0,...,A_\ell)=\{(g_1,...,g_\ell)\in G^\ell\,|\,|\mu(A_0\cap T_{g_1}A_1\cap\cdots \cap T_{g_\ell}A_\ell)-\prod_{j=0}^\ell \mu( A_j)|<\epsilon\}$$
    has uniform density one. 
\end{enumerate}  
\end{thm}
\begin{proof}
(i)$\implies$(ii): For each $j\in\{1,...,\ell\}$, let $\pi_j:G^\ell\rightarrow G$ be defined by $\pi_j(g_1,...,g_\ell)=g_j$. Note that $(T_{\pi_j(\textbf g)})_{\textbf g\in G^\ell}$ is a weakly mixing action and  for any $i\neq j$, $(T_{(\pi_j-\pi_i)(\textbf g)})_{\textbf g\in G^\ell}$ is also weakly mixing. Moreover (see \Cref{6.2.ProductOfWeaklyMixingIsWeaklyMixing}), 
$$(S_{\textbf g})_{\textbf g\in G^\ell}=(\prod_{j=1}^\ell T_{\pi_j(\textbf g)}\times\prod_{i\neq j}T_{(\pi_j-\pi_i)(\textbf g)})_{\textbf g\in G^\ell}$$
is a weakly mixing $G^\ell$-action on the probability space $$(X^{\ell^2},\bigotimes_{j=1}^{\ell^2}\mathcal A,\nu),$$ where 
$\nu=\underbrace{\mu\times\cdots\times \mu}_{\ell^2\text{ times}}$.\\
By \cref{6.2.EquivalentFormsOFWM}, (iii), there exists a set $B\subseteq G^\ell$ with $\overline d_{(F_k)}(B)=0$ such that for any $A_0,A_1\in\bigotimes_{j=1}^{\ell^2}\mathcal A$, 
\begin{equation}\label{6.2.BigProductIsWeaklyMixing}
\lim_{\textbf g\rightarrow\infty,\,\textbf g\not\in B}\nu(A_0\cap S_{\textbf g}A_1)=\nu(A_0)\nu(A_1).
\end{equation}
We start with proving  (ii) for  $m=1$. Let $E\subseteq G^\ell$ with $\overline d_{(F_k)}(E)>0$. By \cref{5.2.FiniteSumsInPositiveDensitySets} (applied to $d=\ell$, $m=1$ and  the set
$(E\setminus B)\subseteq G^\ell$) there exist non-degenerated and essentially distinct sequences $(g_k^{(1)})_{k\in\N}$,...,$(g_k^{(\ell)})_{k\in\N}$ in $G$ with the property that for each $k\in\N$, $\textbf g_k=(g_k^{(1)},...,g_k^{(\ell)})\in E\setminus B$. It follows now from \eqref{6.2.BigProductIsWeaklyMixing} that  $(S_{\textbf g_k})_{k\in\N}$ has the mixing property and hence  for any $j\in\{1,...,\ell\}$, $(T_{g_k^{(j)}})_{k\in\N}$ has the mixing property and for any $i\neq j$, $(T_{g_k^{(j)}-g_k^{(i)}})_{k\in\N}$ has the mixing property as well.\\ 

Assume now that $m>1$. Let $(g_k^{(1)})_{k\in\N}$,...,$(g_k^{(\ell)})_{k\in\N}$ be non-degenerated and essentially distinct sequences in $G$ such that for any distinct $i,j\in\{1,...,\ell\}$, $(T_{g_k^{(j)}})_{k\in\N}$ and $(T_{g_k^{(j)}-g_k^{(i)}})_{k\in\N}$ have the mixing property. Let $(\textbf h_k)_{k\in\N}=(h^{(1)}_k,...,h^{(\ell)}_k)_{k\in\N}$ be a subsequence of $(g^{(1)}_k,...,g_k^{(\ell)})_{k\in\N}$ such that for  any $i,j\in\{1,...,\ell\}$,
\begin{equation}\label{6.2.hkGoingToInfty}
\lim_{\alpha\rightarrow\infty}h^{(j)}_\alpha=\infty\text{ and if }i\neq j,\,\lim_{\alpha\rightarrow\infty}(h^{(j)}_\alpha-h^{(i)}_\alpha)=\infty.
\end{equation}
Observe that, by \eqref{6.2.hkGoingToInfty}, $\{(h^{(1)}_\alpha,...,h^{(\ell)}_\alpha)\,|\,\alpha\in\mathcal F\}$ is an \text{\rm{$\tilde{\text{IP}}$}} set. It follows from our choice of $(g_k^{(1)})_{k\in\N}$,...,$(g_k^{(\ell)})_{k\in\N}$, that for any $M\in\N$, any non-empty set $\alpha\subseteq\{1,...,M\}$, any $A_0,A_1\in\mathcal A$ and any $j\in\{1,...,\ell\}$,
\begin{equation}\label{6.2.DiagonalizingForIP1}
\lim_{k\rightarrow\infty}\mu(T_{-h^{(j)}_\alpha}A_0\cap T_{h^{(j)}_k}A_1)=\mu(A_0)\mu(A_1),
\end{equation}
and for any  $i\neq j$,
\begin{equation}\label{6.2.DiagonalizingForIP2}
\lim_{k\rightarrow\infty}\mu(T_{-(h^{(j)}_\alpha-h^{(i)}_\alpha)}A_0\cap T_{h^{(j)}_k-h^{(i)}_k}A_1)=\mu(A_0)\mu(A_1).
\end{equation}
Passing, if needed, to a subsequence of $(\textbf h_k)_{k\in\N}$, we can derive now from \eqref{6.2.DiagonalizingForIP1} and \eqref{6.2.DiagonalizingForIP2} the following equations
\begin{equation*}\label{6.2.IPlimitForWeaklyMixing}
\mathop{\text{\rm{IP-lim}}}_{\alpha\in\mathcal F}\mu(A_0\cap T_{h^{(j)}_\alpha}A_1)=\mu(A_0)\mu(A_1),
\end{equation*}
and if $i\neq j$,
\begin{equation*}
\mathop{\text{\rm{IP-lim}}}_{\alpha\in\mathcal F}\mu(A_0\cap T_{h^{(j)}_\alpha-h^{(i)}_\alpha}A_1)=\mu(A_0)\mu(A_1).
\end{equation*}
We can conclude now the proof of (i)$\implies$(ii) by arguing as in the proof of \cref{5.2.FiniteSumsInPositiveDensitySets} and imitating the constructions in the proofs of \cref{5.2.IPPoincare} and \cref{5.1.SigmaInEveryIP}.\\  
(ii)$\implies$(iii): This follows from \cref{2.MainResult}.\\
(iii)$\implies$(iv): Let $E=G^\ell\setminus R_\epsilon(A_0,...,A_\ell)$. It suffices to show that for any F{\o}lner sequence $(F_k)_{k\in\N}$ in $G^\ell$, $\overline d_{(F_k)}(E)=0$. To see this, note that  if this was not the case, (iii) would imply that $E\cap R_\epsilon(A_0,...,A_\ell)\neq \emptyset$, a contradiction.\\
(iv)$\implies$(i): This implication is trivial and is omitted. 
\end{proof}
We conclude this section with a corollary of \cref{6.2.GlobalWeaklyMixing} which has diagonal nature (This corollary can also be obtained   from the main result in \cite{BerRosJointErgodicity}). 
\begin{cor}\label{6.2.DiagonalDensityResult}
Let $(G,+)$ be a countable abelian group, let $(X,\mathcal A,\mu,(T_g)_{g\in G})$ be a measure preserving system and let $\phi_1,...,\phi_\ell:G\rightarrow G$ be homomorphisms with the property that for any $j\in\{1,...,\ell\}$,  $(T_{\phi_j(g)})_{g\in G}$ is weakly mixing and  for any $i\neq j$,  $(T_{(\phi_j-\phi_i)(g)})_{g\in G}$ is also weakly mixing. For any F{\o}lner sequence $(F_k)_{k\in\N}$ in $G$ and any $A_0,...,A_\ell\in\mathcal A$,
\begin{equation}\label{6.2.DensityLimit2}
\lim_{k\rightarrow\infty}\frac{1}{|F_k|}\sum_{g\in F_k}|\mu(A_0\cap T_{\phi_1(g)}A_1\cap\cdots\cap T_{\phi_\ell(g)}A_\ell)-\prod_{j=0}^\ell \mu(A_j)|=0.
\end{equation}
\end{cor}
\begin{proof}
By \cref{6.2.LimitEquivalenceWM}, in order to prove \eqref{6.2.DensityLimit2}, it suffices to show that for any $E\subseteq G$ with $\overline d_{(F_k)}(E)>0$, there exists a non-degenerated sequence $(\textbf g_k)_{k\in\N}=(g_{k,1},...,g_{k,\ell})_{k\in\N}$ in $G^\ell$ with $\{g_\alpha\,|\,\alpha\in\N^{(\ell)}\}\subseteq E$ such that 
\begin{equation}\label{6.2.PreDensityDiagonalLimit}
\rlim{\alpha\in\N^{(\ell)}}\mu(A_0\cap T_{\phi_1(g_\alpha)}A_1\cap\cdots\cap T_{\phi_\ell(g_\alpha)}A_\ell)=\prod_{j=0}^\ell \mu(A_j).
\end{equation}
By  \cref{6.2.GlobalWeaklyMixing}, (ii), applied to the weakly mixing $G$-action
$$(S_g)_{g\in G}=(\prod_{j=1}^\ell T_{\phi_j(g)}\times\prod_{i\neq j}T_{(\phi_j-\phi_i)(g)})_{g\in G},$$ there exists a non-degenerated sequence $(g_{k,1},...,g_{k,\ell})_{k\in\N}$ in $G$, with $\{g_\alpha\,|\,\alpha\in\N^{(\ell)}\}\subseteq E$, and such that for any $t\in\{1,...,\ell\}$, the sequence $(S_{g_{k,t}})_{k\in\N}$ has the mixing property. It follows that for any $t\in\{1,...,\ell\}$ and any $j\in\{1,...,\ell\}$, $(T_{\phi_j(g_{k,t})})_{k\in\N}$ has the mixing property and for any $t$ and $i\neq j$, $(T_{(\phi_j-\phi_i)(g_{k,t})})_{k\in\N}$ has the mixing property as well.  The result now follows from  \cref{2.MainResult}.
\end{proof}
\begin{rem}
Taking in \cref{6.2.DiagonalDensityResult} $G=\Z$, one obtains the following classical result due to Furstenberg (Cf. Theorem 4.11 in \cite{FBook}):\\
\begin{adjustwidth}{0.5cm}{0.5cm}
\textit{For any weakly mixing system $(X,\mathcal A,\mu, T)$, any non-zero and distinct integers  $a_1,...,a_\ell$  and any $A_0,...,A_\ell\in\mathcal A$,
$$\lim_{N-M\rightarrow\infty}\frac{1}{N-M}\sum_{n=M+1}^N|\mu(A_0\cap T^{a_1n}A_1\cap \cdots\cap T^{a_\ell n}A_\ell)-\prod_{j=0}^\ell\mu(A_j)|=0.$$}
\end{adjustwidth}
\end{rem}
\bibliography{Bib.bib}
\bibliographystyle{plain}

\bigskip
\footnotesize

\noindent
Vitaly Bergelson\\
\textsc{Department of Mathematics, The Ohio State University, Columbus, OH 43210, USA}\par\nopagebreak
\noindent
\href{mailto:vitaly@math.ohio-state.edu}
{\texttt{vitaly@math.ohio-state.edu}}

\medskip

\noindent
Rigoberto Zelada\\
\textsc{Department of Mathematics, The Ohio State University, Columbus, OH 43210, USA}\par\nopagebreak
\noindent
\href{mailto:zeladacifuentes.1@osu.edu}
{\texttt{zeladacifuentes.1@osu.edu}}
\end{document}